\documentclass[a4paper,reqno]{amsart}
\usepackage[all]{xy} 
\usepackage{amssymb} 
\usepackage{hyperref}
\usepackage{eucal}

\numberwithin{equation}{section}

\newtheorem{definition}{Definition}[section]
\newtheorem{lemma}[definition]{Lemma}
\newtheorem{theorem}[definition]{Theorem}
\newtheorem{proposition}[definition]{Proposition}
\newtheorem{corollary}[definition]{Corollary}
\newtheorem{remarkth}[definition]{Remark}
\newtheorem{example}[definition]{Example}
\newenvironment{remark}{\begin{remarkth}\upshape}{\hfill$\diamond$\end{remarkth}}

\renewcommand{\emph}[1]{{\bfseries\itshape{#1}}}

\newcommand{\R}{\mathbb{R}} 
\newcommand{\N}{\mathbb{N}} 
\newcommand{\C}{\mathcal{C}}

\newcommand{\lcf}{\lbrack\! \lbrack}
\newcommand{\rcf}{\rbrack\! \rbrack}

\newcommand{\lvec}[1]{\overleftarrow{#1}}

\newcommand{\qquand}{\qquad\text{and}\qquad}
\newcommand{\quand}{\quad\text{and}\quad}

\makeatletter
\newcommand\prol{\@ifstar{\@proldf}{\@prolpf}} 
\def\@prolpf{\@ifnextchar[{\@prolpf@wrt}{\@prolpf@}}
\def\@prolpf@wrt[#1]#2{\@ifnextchar[{\@prolpf@wrt@at{#1}{#2}}{\@prolpf@wrt@{#1}{#2}}}

\def\@prolpf@wrt@at#1#2[#3]{\prolsymbol^{#1}_{#3}#2}
\def\@prolpf@wrt@#1#2{\prolsymbol^{#1}#2}
\def\@prolpf@#1{\@ifnextchar[{\@prolpf@at{#1}}{\@prolpf@@{#1}}}
\def\@prolpf@at#1[#2]{\prolsymbol_{#2}#1}
\def\@prolpf@@#1{\prolsymbol#1}
\def\@proldf{\@ifnextchar[{\@proldf@wrt}{\@proldf@}}
\def\@proldf@wrt[#1]#2{\@ifnextchar[{\@proldf@wrt@at{#1}{#2}}{\@proldf@wrt@{#1}{#2}}}

\def\@proldf@wrt@at#1#2[#3]{\prolsymbol^{*#1}_{#3}#2}
\def\@proldf@wrt@#1#2{\prolsymbol^{*#1}#2}
\def\@proldf@#1{\@ifnextchar[{\@proldf@at{#1}}{\@proldf@@{#1}}}
\def\@proldf@at#1[#2]{\prolsymbol^*_{#2}#1}
\def\@proldf@@#1{\prolsymbol^*#1}
\def\prolsymbol{\mathcal{L}}
\makeatother
\newcommand{\TEE}[1][]{\mathcal{T}^E_{#1}E}

\newcommand{\X}{\mathcal{X}}
\newcommand{\Y}{\mathcal{Y}}
\newcommand{\V}{\mathcal{V}}
\newcommand{\T}{\mathcal{T}}
\renewcommand{\P}{\mathcal{P}}

\def\lcf{\lbrack\! \lbrack}
\def\rcf{\rbrack\! \rbrack}
\newcommand{\cinfty}[1]{C^\infty(#1)}
\newcommand{\set}[2]{\left\{\,#1\left.\vphantom{#1#2}\,\right\vert\,#2\,
\right\}}

\newcommand{\pd}[2]{\frac{\partial #1}{\partial #2}}

\newcommand{\Sec}[1]{\operatorname{Sec}(#1)}

\newcommand{\sode}{{\textsc{sode}}}

\begin{document}

\title[Constrained mechanics on Lie algebroids]
{Singular lagrangian systems and variational constrained  mechanics on Lie algebroids}

\author[D. Iglesias]{D. Iglesias}
\address{D.\ Iglesias:
Instituto de Matem\'aticas y F{\'\i}sica Fundamental, Consejo
Superior de Investigaciones Cient\'{\i}ficas, Serrano 123, 28006
Madrid, Spain} \email{iglesias@imaff.cfmac.csic.es}

\author[J.\ C.\ Marrero]{J.\ C.\ Marrero}
\address{J.\ C.\ Marrero:
Departamento de Matem\'atica Fundamental, Facultad de
Ma\-te\-m\'a\-ti\-cas, Universidad de la Laguna, La Laguna,
Tenerife, Canary Islands, Spain} \email{jcmarrer@ull.es}

\author[D.\ Mart\'{\i}n de Diego]{D. Mart\'{\i}n de Diego}
\address{D.\ Mart\'{\i}n de Diego:
Instituto de Matem\'aticas y F{\'\i}sica Fundamental, Consejo
Superior de Investigaciones Cient\'{\i}ficas, Serrano 123, 28006
Madrid, Spain} \email{d.martin@imaff.cfmac.csic.es}

\author[D.\ Sosa]{D.\ Sosa}
\address{D.\ Sosa:
Departamento de Econom\'{\i}a Aplicada, Facultad de CC. EE. y
Empresariales, Universidad de La Laguna, La Laguna, Tenerife, Canary
Islands, Spain} \email{dnsosa@ull.es}



\begin{abstract}
The purpose of this paper is describe Lagrangian Mechanics for
constrained systems on Lie algebroids, a natural framework which
covers a wide range of situations (systems on Lie groups,
quotients by the action of a Lie group, standard tangent
bundles...). In particular, we are interested in two cases:
\emph{singular Lagrangian systems} and \emph{vakonomic mechanics}
(variational constrained mechanics). Several examples illustrate
the interest of these developments.
\end{abstract}

\maketitle

\tableofcontents
\section{Introduction}

There is a vast literature around the Lagrangian formalism in
mechanics justified by the central role played by these systems in
the foundations of modern mathe\-ma\-tics and physics. In some
interesting systems some problems often arise  due to their
singular nature that give rise to the presence of constraints
manifesting that the evolution problem is not well posed
(\emph{internal constraints}). Constraints can also manifest a
priori restrictions on the states of the system which are often
imposed either by physical arguments or by external conditions
(\emph{external constraints}). Both cases are of considerable
importance.

Systems with \emph{internal constraints} are quite interesting
since many dynamical systems are given in terms of non-symplectic
forms instead of the more habitual symplectic ones. The more
frequent case appears in the Lagrangian formalism of singular
mechanical systems which are a commonplace in many physical
theories (as in Yang-Mills theories, gravitation, etc). Also, in
systems that appears as a limit (as in Chern-Simons lagrangians);
for instance consider the following lagrangian
\[
L=\frac{1}{2} m_i (\dot{q}^i)^2 +e A_i(q)\dot{q}^i-V(q)
\]
in the limit $m_i\rightarrow 0$ for some $i$. In other cases, it
is necessary work with a new singular lagrangian in a extended
space since the original lagrangian is ``ill defined" (only
locally defined on the original space), as it happens for the
electron monopole system (see \cite{AC}).

Another motivation for the present work is the study of lagrangian
systems  subjected to \emph{external constraints} (holonomic and
nonholonomic) \cite{Bl}. These systems have  a wide application in
many different areas: engineering, optimal control theory,
mathematical economics (growth economic theory), subriemannian
geometry, motion of microorganisms, etc.

Constrained variational calculus have a rich geometric structure.
Many of these systems usually exhibit invariance under the action
of a Lie group of symmetries and they can be notably simplified
using their symmetric properties by reducing the degrees of
freedom of the original system. In previous studies it is imposed
a separate study for each class of systems since the lack of a
unified framework for dealing simultaneously with all the systems.

 Recent investigations have lead to a unifying geometric framework to
covering these plethora of particular situations. It is precisely
the underlying structure of a Lie algebroid on the phase space
which allows a unified treatment. This idea was first introduced
by Weinstein \cite{weinstein} in order to define a Lagrangian
formalism which is general enough to account for the various types
of systems. The geometry and dynamics on Lie algebroids have been
extensively studied during the past years. In particular, in
\cite{mart}, E. Mart{\'\i}nez developed a geometric formalism of
mechanics on Lie algebroids similar to Klein's formalism of the
ordinary Lagrangian mechanics and, more recently, a description of
the Hamiltonian dynamics on a Lie algebroid was given
in~\cite{LMM,Medina}. The key concept in this theory is the
prolongation, ${\mathcal T }^E E$, of the Lie algebroid over the
fibred projection $\tau$ (for the Lagrangian formalism) and the
prolongation, ${\mathcal T}^E E^*$, over the dual fibred
projection $\tau^*: E^*\to Q$ (for the Hamiltonian formalism). See
\cite{LMM} for more details. Of course, when the Lie algebroid is
$E=TQ$ we obtain that ${\mathcal T}^E E=T(TQ)$ and ${\mathcal
T}^EE^*=T(T^*Q)$, recovering the classical case. Another approach
to the theory was discussed in \cite{GGU}. The existence of
symmetries in these systems makes interesting to generalize the
Gotay-Nester-Hinds algorithm \cite{GNH} to the case of Lie
algebroids with a presymplectic section. These results are easily
extended to the case of implicit differential equations on Lie
algebroids.

The second author and collaborators analyzed the case of
nonholonomic me\-cha\-nics on Lie algebroids \cite{CoLeMaMa}. Now,
we also pretend to study singular Lagrangian systems and vakonomic
mechanics on Lie algebroids (obtained through the application of a
constrained variational principle).

The paper is organized as follows. In Section \ref{algebroides},
we recall the notion of a Lie algebroid and several aspects
related with it. In particular, we describe the prolongation
${\mathcal T}^EE$ of a Lie algebroid $E$ over the projection $\tau
:E\to Q$ and how to use this construction to develop Lagrangian
mechanics on a Lie algebroid $E$ with a Lagrangian function
$L:E\to \R$, introducing several important objects, such as the
Lagrangian energy $E_L$ and the Cartan 2-section $\omega _L$. When
the Lagrangian function is regular (that is, $\omega _L$ is
nondegenerate), we have existence and uniqueness of solutions for
the Euler-Lagrange equations. However, if the Lagrangian is
singular (or degenerate) we cannot guarantee these results.
Motivated by this fact, in Section \ref{pre} we introduce a
constraint algorithm for presymplectic Lie algebroids which
generalizes the well-known Gotay-Nester-Hinds algorithm. In
addition, we show that a Lie algebroid morphism which relates two
presymplectic Lie algebroids induces a relation between the two
associated constraint algorithms.

In Section \ref{Sec:singular}, we apply the results of Section
\ref{pre} to singular Lagrangian systems on Lie algebroids. More
precisely, given a Lie algebroid $\tau :E\to Q$ and a singular
Lagrangian function $L:E\to \R$, we look for a solution $X$ of the
presymplectic system $({\mathcal T}^E E, \omega _L, d^{{\mathcal
T}^EE}E_L)$ which is a SODE along the final constraint
submanifold. An example for an Atiyah algebroid illustrates our
theory.

In Section \ref{sec:vak}, we develop a geometric description of
vakonomic mechanics on Lie algebroids. In this setting, given a
Lie algebroid $\tau :E\to Q$, we have a pair $(L,M)$ where $L$ is
a Lagrangian function on $E$ and $M\subseteq E$ is a constraint
submanifold. In Section \ref{subsec:vak-bracket}, we deduce the
vakonomic equations using our constraint algorithm and study the
particular case when it stops in the first step. In this
situation, if the restriction of the presymplectic 2-section to
the final constraint algebroid is symplectic, one can introduce
the vakonomic bracket, which allows us to give the evolution of
the observables. On the other hand, it is well know that classical
vakonomic systems can be obtained from a constrained variational
principle. This can also be done for vakonomic systems on Lie
algebroids, as it is shown in Section \ref{subsec:variational}. In
the particular case when we do not have constraints, our approach
can be seen as the Skinner-Rusk formulation of Lagrangian
Mechanics on Lie algebroids. This is explained in Section
\ref{subsec:examples}, where we also illustrate our results with
several interesting examples: If $E$ is the standard Lie algebroid
$\tau_{TQ}: TQ\to Q$, then we recover some well-known results (see
\cite{CLMM}) for vakonomic systems; in the case when $E=\mathfrak
g$, a real Lie algebra of finite dimension, we are able to model a
certain class of Optimal Control problems on Lie groups (see
\cite{KM,Kr}); for the Atiyah algebroid, we analyze some problems
related with reduction in subriemannian geometry by means of the
non-holonomic connection \cite{BKMM}. Finally, we study  Optimal
Control on Lie algebroids  as vakonomic systems and, as an
illustration of our techniques, we find the equations of motion of
the Plate-Ball system.

\section{Lie algebroids} \label{algebroides}
Let $E$ be a vector bundle of rank $n$ over a manifold $Q$ of
dimension $m$ and $\tau:E\to Q$ be the vector bundle projection.
Denote by $\Gamma(E)$ the $C^\infty(Q)$-module of sections of
$\tau:E\to Q$. A \emph{Lie algebroid structure }
$(\lcf\cdot,\cdot\rcf,\rho)$ on $E$ is a Lie bracket
$\lcf\cdot,\cdot\rcf$ on the space $\Gamma(E)$ and a bundle map
$\rho:E\to TQ$, called \emph{the anchor map}, such that if we also
denote by $\rho:\Gamma(E)\to {\frak X}(Q)$ the homomorphism of
$C^\infty(Q)$-modules induced by the anchor map, then
\[
\lcf X,fY\rcf=f\lcf X,Y\rcf + \rho(X)(f)Y,
\]
for $X,Y\in \Gamma(E)$ and $f\in C^\infty(Q)$. The triple
$(E,\lcf\cdot,\cdot\rcf,\rho)$ is called \emph{a Lie algebroid
over} $Q$ (see \cite{Ma}).

If $(E,\lcf\cdot,\cdot\rcf,\rho)$ is a Lie algebroid over $Q,$ then
the anchor map $\rho:\Gamma(E)\to {\frak X}(Q)$ is a homomorphism
between the Lie algebras $(\Gamma(E),\lcf\cdot,\cdot\rcf)$ and
$({\frak X}(Q),[\cdot,\cdot])$.

Standard examples of Lie algebroids are real Lie algebras of
finite dimension and the tangent bundle $TQ$ of an arbitrary
manifold $Q.$

Another example of a Lie algebroid may be constructed as follows.
Let $\pi:P\to Q$ be a principal bundle with structural group $G$.
Denote by $\Phi:G\times P\to P$ the free action of $G$ on $P$ and
by $T\Phi:G\times TP\to TP$ the tangent action of $G$ on $TP$.
Then, one may consider the quotient vector bundle
$\tau_P|G:TP/G\to Q=P/G$ and the sections of this vector bundle
may be identified with the vector fields on $P$ which are
invariant under the action $\Phi$. Using that every $G$-invariant
vector field on $P$ is $\pi$-projectable and the fact that the
standard Lie bracket on vector fields is closed with respect to
$G$-invariant vector fields, we can induce a Lie algebroid
structure on $TP/G$. The resultant Lie algebroid is called
\emph{the Atiyah (gauge) algebroid associated with the principal
bundle} $\pi:P\to Q$ (see \cite{LMM,Ma}).

If $(E,\lcf\cdot,\cdot\rcf,\rho)$ is a Lie algebroid, one may
define \emph{the differential of $E$}, $d^E:\Gamma(\wedge^k
E^*)\to \Gamma(\wedge^{k+1}E^*)$, as follows
$$\begin{array}{l} d^E \mu(X_0,\dots, X_k)=\displaystyle
\sum_{i=0}^{k} (-1)^i\rho(X_i)(\mu(X_0,\dots,
\widehat{X_i},\dots, X_k))\\[10pt] \hspace{2.75cm}+\displaystyle\sum_{i<
j}(-1)^{i+j}\mu(\lcf X_i,X_j\rcf,X_0,\dots,
\widehat{X_i},\dots,\widehat{X_j},\dots ,X_k),
\end{array}$$
for $\mu\in \Gamma(\wedge^k E^*)$ and $X_0,\dots ,X_k\in
\Gamma(E).$ It follows that $(d^E)^2=0$. Moreover, if
$X\in\Gamma(E)$, one may introduce, in a natural way, \emph{the
Lie derivative with respect to $X$,} as the operator ${\mathcal
L}^E_X:\Gamma(\wedge^kE^*)\to \Gamma(\wedge^k E^*)$ given by
${\mathcal L}^E_X=i_X\circ d^E + d^E\circ i_X.$

Note that if $E = TQ$ and $X \in \Gamma(E) = {\frak X}(Q)$ then
$d^{TQ}$ and ${\mathcal L}_{X}^{TQ}$ are the usual differential and
the usual Lie derivative with respect to $X$, respectively.

If we take local coordinates $(x^i)$ on $Q$ and a local basis
$\{e_A\}$ of sections of $E$, then we have the corresponding local
coordinates $(x^i,y^A)$ on $E$, where $y^A(e)$ is the $A$-th
coordinate of $e\in E$ in the given basis. Such coordinates
determine local functions $\rho_A^i$, ${\mathcal C}_{A B}^{C}$ on
$Q$ which contain the local information of the Lie algebroid
structure and, accordingly, they are called \emph{the structure
functions of the Lie algebroid.} They are given by
\[
\rho(e_A)=\rho_A^i\frac{\partial }{\partial x^i}\;\;\;\mbox{ and
}\;\;\; \lcf e_A,e_B\rcf={\mathcal C}_{A B}^C e_C.
\]
These functions should satisfy the relations
$$\rho_A^j\frac{\partial \rho_B^i}{\partial x^j}
-\rho_B^j\frac{\partial \rho_A^i}{\partial x^j}= \rho_C^i{\mathcal
C}_{AB}^C ,$$
$$\sum_{cyclic(A,B,C)}\Big ( \rho_{A}^i\frac{\partial {\mathcal
C}_{BC}^D}{\partial x^i} + {\mathcal C}_{A F}^D {\mathcal
C}_{BC}^F \Big )=0,$$
which are usually called \emph{the structure equations}.

If $f\in C^\infty(Q)$, we have that
$$d^E f=\frac{\partial f}{\partial x^i}\rho_A^i e^A,$$
where $\{e^A\}$ is the dual basis of $\{e_A\}$. On the other hand,
if $\theta\in \Gamma(E^*)$ and $\theta=\theta_C e^C$ it follows that
$$d^E \theta=(\frac{\partial \theta_C}{\partial
x^i}\rho^i_B-\frac{1}{2}\theta_A {\mathcal C}^A_{BC})e^{B}\wedge
e^C.$$
In particular, \[d^E x^i=\rho_A^ie^A,\;\;\; d^E
e^A=-\frac{1}{2}{\mathcal C}_{B C}^{A} e^B\wedge e^C.\]

On the other hand, if $(E,\lcf\cdot,\cdot \rcf,\rho)$ and
$(E',\lcf\cdot,\cdot\rcf', \rho')$ are Lie algebroids over $Q$ and
$Q'$, respectively, then a morphism of vector bundles $(F,f)$ from
$E$ to $E'$

\begin{picture}(375,90)(80,20)
\put(195,20){\makebox(0,0){$Q$}}
\put(250,25){$f$}\put(210,20){\vector(1,0){80}}
\put(310,20){\makebox(0,0){$Q'$}} \put(185,50){$\tau$}
\put(195,70){\vector(0,-1){40}} \put(320,50){$\tau'$}
\put(310,70){\vector(0,-1){40}} \put(195,80){\makebox(0,0){$E$}}
\put(250,85){$F$}\put(210,80){\vector(1,0){80}}
\put(310,80){\makebox(0,0){$E'$}} \end{picture}

\vspace{.5cm} \noindent is a \emph{Lie algebroid morphism} if
\begin{equation}\label{Qorph}
d^E ((F,f)^*\phi')= (F, f)^*(d^{E'}\phi'), \;\;\; \mbox{ for
}\phi'\in \Gamma(\wedge^k(E')^*).
\end{equation}
Note that $(F, f)^*\phi'$ is the section of the vector bundle
$\wedge^kE^*\to Q$ defined by
\[
((F,f)^*\phi')_x(a_1,\dots ,a_k)=\phi'_{f(x)}(F(a_1),\dots ,F(a_k)),
\]
for $x\in Q$ and $a_1,\dots ,a_k\in E_{x}$. We remark that
(\ref{Qorph}) holds if and only if
$$\begin{array}{l} d^E(g'\circ f)=(F, f)^{*}(d^{E'}g'), \;\;\;
\mbox{for }g'\in
C^\infty(Q'),\\[5pt] d^E((F,f)^*\alpha')=(F,
f)^*(d^{E'}\alpha'),\;\;\;\mbox{for } \alpha'\in \Gamma((E')^*).
\end{array}$$

If $(F,f)$ is a Lie algebroid morphism, $f$ is an injective
immersion and $F_{|E_x}:E_x\rightarrow E'_{f(x)}$ is injective,
for all $x\in Q$, then $(E,\lcf\cdot,\cdot\rcf_{E},\rho_{E})$ is
said to be a \emph{Lie subalgebroid} of
$(E',\lcf\cdot,\cdot\rcf_{E'},\rho_{E'})$.

If $Q=Q'$ and $f=id:Q\to Q$ then, it is easy prove that the pair
$(F,id)$ is a Lie algebroid morphism if and only if
\[F\lcf X,Y\rcf=\lcf FX,FY \rcf',\;\;\; \rho'(FX)=\rho(X),\]
for $X,Y\in \Gamma(E).$

Let $(E, \lcf \cdot , \cdot \rcf, \rho)$ be a Lie algebroid over a
manifold $Q$ and $E^{*}$ be the dual bundle to $E$. Then, $E^{*}$
admits a \emph{linear Poisson structure} $\{ \cdot , \cdot
\}_{E^{*}}$, that is,
\[
\{ \cdot , \cdot \}_{E^{*}}: C^{\infty}(E^{*}) \times
C^{\infty}(E^{*}) \to C^{\infty}(E^{*})
\]
is a $\R$-bilinear map,

$
\begin{array}{ll}
\{F, G \}_{E^{*}} = -\{G, F \}_{E^{*}},
&\mbox{(skew-symmetry),}\\[5pt]
\{F F', G \}_{E^{*}} = F \{F', G \}_{E^{*}}
+ F' \{F, G \}_{E^{*}},& \mbox{(the Leibniz rule),}\\[5pt]
\{F, \{G, H \}_{E^{*}} \}_{E^{*}} + \{G, \{H, F \}_{E^{*}}
\}_{E^{*}} &\\[5pt] + \{H, \{F, G \}_{E^{*}} \}_{E^{*}} = 0,
&\mbox{(the Jacobi identity),}
\end{array}
$

\noindent for $F, F', G, H \in C^{\infty}(E^{*})$ and, in
addition,
\[
P, P' \mbox{ linear functions on } E^{*}  \Rightarrow \{P,
P'\}_{E^{*}} \mbox{ is a linear function on } E^{*}.
\]
If $(x^{i})$ are local coordinates on an open subset $U$ of $Q$,
$\{e_{A}\}$ is a local basis of $\Gamma(E)$ on $U$ and $F, G \in
C^{\infty}(E^{*})$ then the local expression of the Poisson
bracket of $F$ and $G$ is
\begin{equation}\label{Poisson}
\{F, G\}_{E^{*}} = \rho^{i}_{A}\left( \displaystyle \frac{\partial
F}{\partial x^{i}} \frac{\partial G}{\partial p_{A}} -
\frac{\partial F}{\partial p_{A}} \frac{\partial G}{\partial
x^{i}}\right)  - {\mathcal C}_{AB}^{C} p_{C} \frac{\partial
F}{\partial p_{A}} \frac{\partial G}{\partial p_{B}},
\end{equation}
where $(x^{i}, p_{A})$ are the corresponding coordinates on
$E^{*}$ (for more details, see \cite{LMM}).

\subsection{The prolongation of a Lie algebroid over a fibration}

Let $(E,\lcf\cdot,\cdot\rcf,\rho)$ be a Lie algebroid of rank $n$
over a manifold $Q$ of dimension $m$ and $\pi : P \to Q$ be a
fibration, that is, a surjective submersion.

We consider the subset ${\mathcal T}^EP$ of $E\times TP$ defined
by ${\mathcal T}^EP=\displaystyle \bigcup_{p\in P} {\mathcal
T}^E_pP$, where
\[
{\mathcal T}^E_pP=\{(b,v)\in E_{\pi(p)}\times
T_{p}P\,|\,\rho(b)=(T_{p}\pi)(v)\}
\]
and $T\pi:TP\to TQ$ is the tangent map to $\pi$.

Denote by $\tau^{\pi}:{\mathcal T}^EP\to P$ the map given by
\[
\tau^{\pi}(b,v)=\tau_{P}(v),
\]
for $(b,v)\in {\mathcal T}^EP,$ $\tau_{P}:TP\to P$ being the
canonical projection. Then, if $m'$ is the dimension of $P$, one
may prove that
\[
\dim\,{\mathcal T}^E_pP=n+m'-m.
\]
Thus, we conclude that ${\mathcal T}^EP$ is a vector bundle over
$P$ of rank $n+m'-m$ with vector bundle projection
$\tau^{\pi}:{\mathcal T}^EP\to P.$

A section $\tilde{X}$ of $\tau^{\pi}: {\mathcal T}^{E}P \to P$ is
said to be \emph{projectable} if there exists a section $X$ of
$\tau: E \to Q$ and a vector field $U$ on $P$ which is
$\pi$-projectable to the vector field $\rho(X)$ and such that
$\tilde{X}(p) = (X(\pi(p)), U(p))$, for all $p \in P$. For such a
projectable section $\tilde{X}$, we will use the following
notation $\tilde{X} \equiv (X, U)$. It is easy to prove that one
may choose a local basis of projectable sections of the space
$\Gamma({\mathcal T}^{E}P)$.

The vector bundle $\tau^{\pi}: {\mathcal T}^{E}P \to P$ admits a
Lie algebroid structure $(\lcf \cdot , \cdot \rcf^{\pi},
\rho^{\pi})$. In fact,
\[
\lcf (X_{1}, U_{1}), (X_{2}, U_{2})\rcf^{\pi} = (\lcf X_{1},
X_{2}\rcf, [U_{1}, U_{2}]), \; \; \rho^{\pi}(X_{1}, U_{1}) =
U_{1}.
\]
The Lie algebroid $({\mathcal T}^{E}P, \lcf \cdot , \cdot
\rcf^{\pi}, \rho^{\pi})$ is called \emph{the prolongation of $E$
over $\pi$ or the $E$-tangent bundle to $P$}. Note that if
$pr_1:{\mathcal T}^EP\to E$ is the canonical projection on the
first factor, then the pair $(pr_1,\pi)$ is a morphism between the
Lie algebroids $({\mathcal
T}^EP,\lcf\cdot,\cdot\rcf^{\pi},\rho^{\pi})$ and
$(E,\lcf\cdot,\cdot \rcf,\rho)$ (for more details, see
\cite{LMM}).

\begin{example}
{\rm Let $(E,\lcf\cdot,\cdot \rcf,\rho)$ be a Lie algebroid of
rank $n$ over a manifold $Q$ of dimension $m$ and $\tau:E\to Q$ be
the vector bundle projection. Consider the prolongation ${\mathcal
T}^EE$ of $E$ over $\tau,$
\[
{\mathcal T}^EE=\{({e},v)\in E\times TE\,|\,
\rho({e})=(T\tau)(v)\}.
\]
${\mathcal T}^EE$ is a Lie algebroid over $E$ of rank $2n$ with
Lie algebroid structure $(\lcf\cdot,\cdot\rcf^{\tau},
\rho^{\tau})$.

If $(x^i)$ are local coordinates on an open subset $U$ of $Q$ and
$\{e_A\}$ is a basis of sections of the vector bundle
$\tau^{-1}(U)\to U$, then $\{{\mathcal X}_A,{\mathcal V}_A\}$ is a
basis of sections of the vector bundle
$(\tau^{\tau})^{-1}(\tau^{-1}(U))\to \tau^{-1}(U)$, where
$\tau^{\tau}:{\mathcal T}^EE\to E$ is the vector bundle projection
and
\[
\begin{array}{rcl}
{\mathcal
X}_A(e)&=&(e_A(\tau(e)),\rho_A^i\displaystyle\frac{\partial
}{\partial x^i}_{|e}),\\
{\mathcal V}_A(e)&=&(0,\displaystyle\frac{\partial }{\partial
y^{A}}_{|e}),
\end{array}
\]
for $e\in \tau^{-1}(U).$ Here, $\rho_A^i$ are the components of
the anchor map with respect to the basis $\{e_A\}$ and $(x^i,y^A)$
are the local coordinates on $E$ induced by the local coordinates
$(x^i)$ and the basis $\{e_A\}$. Using the local basis
$\{{\mathcal X}_A,{\mathcal V}_A\}$, one may introduce, in a
natural way, local coordinates $(x^i,y^A;z^A,v^A)$ on ${\mathcal
T}^EE.$ On the other hand, we have that
\[
\begin{array}{rcl}
\rho^{\tau}({\mathcal X}_A)&=&\rho_A^i\displaystyle\frac{\partial
}{\partial x^i},\;\;\; \rho^{\tau}({\mathcal
V}_A)=\displaystyle\frac{\partial
}{\partial y^A},\\[8pt]
\lcf{\mathcal X}_A,{\mathcal X}_B\rcf^{\tau}&=&{\mathcal
C}_{AB}^C{\mathcal
X}_C,\\[8pt]
\lcf{\mathcal X}_A,{\mathcal V}_B\rcf^{\tau}&=&\lcf {\mathcal V}_A,
{\mathcal V}_B\rcf^{\tau}=0,
\end{array}
\]
for all $A$ and $B$, ${\mathcal C}_{AB}^C$ being the structure
functions of the Lie bracket $\lcf\cdot,\cdot\rcf$ with respect to
the basis $\{e_A\}$.

The vector subbundle $({\mathcal T}^EE)^V$ of ${\mathcal T}^EE$
whose fibre at the point $e\in E$ is
$$({\mathcal T}^E_eE)^V=\{(0,v)\in E\times T_eE/\,(T_e\tau)(v)=0\}$$
is called \emph{the vertical subbundle}. Note that $({\mathcal
T}^EE)^V$ is locally generated by the sections $\{{\mathcal
V}_A\}$.

Two canonical objects on ${\mathcal T}^EE$ are \emph{the Euler
section} $\Delta$ and \emph{the vertical endomorphism $S$}. $\Delta$
is the section of ${\mathcal T}^EE\to E$ locally defined by
\begin{equation}
\label{Lioulo} \Delta = y^{A}\V_{A},
\end{equation}
and $S$ is the section of the vector bundle $({\mathcal
T}^EE)\otimes ({\mathcal T}^EE)^*\to E$ locally characterized by
the following conditions
\begin{equation}\label{endverlo}
S\X_{A} = \V_{A}, \makebox[.3cm]{} S\V_{A} = 0, \makebox[.3cm]{}
\mbox{ for all } A.
\end{equation}
Finally, a section $\xi$ of ${\mathcal T}^EE\to E$ is said to be a
\emph{second order differential equation} (SODE) on $E$ if
$S(\xi)=\Delta$ or, alternatively, $pr_1(\xi(e))=e$, for all $e\in
E$ (for more details, see \cite{LMM}). }
\end{example}

\begin{example}
{\rm

Let $(E,\lcf\cdot,\cdot \rcf,\rho)$ be a Lie algebroid of rank $n$
over a manifold $Q$ of dimension $m$ and $\tau^*:E^*\to Q$ be the
vector bundle projection of the dual bundle $E^*$ to $E$.

We consider the prolongation ${\mathcal T}^EE^*$ of $E$ over
$\tau^*,$
\[
{\mathcal T}^EE^*=\{({e}',v)\in E\times TE^*\,|\,
\rho({e}')=(T\tau^*)(v)\}.
\]
${\mathcal T}^EE^*$ is a Lie algebroid over $E^*$ of rank $2n$
with Lie algebroid structure $(\lcf\cdot,\cdot\rcf^{\tau^*},
\rho^{\tau^*})$.

If $(x^i)$ are local coordinates on an open subset $U$ of $Q$,
$\{e_A\}$ is a basis of sections of the vector bundle
$\tau^{-1}(U)\to U$ and $\{e^A\}$ is the dual basis of $\{e_A\}$,
then $\{{\mathcal Y}_A,{\mathcal P}^A\}$ is a basis of sections of
the vector bundle $(\tau^{\tau^*})^{-1}((\tau^*)^{-1}(U))\to
(\tau^*)^{-1}(U)$, where $\tau^{\tau^*}:{\mathcal T}^EE^*\to E^*$
is the vector bundle projection and
\[
\begin{array}{rcl}
{\mathcal
Y}_A(e^*)&=&(e_A(\tau^*(e^*)),\rho_A^i\displaystyle\frac{\partial
}{\partial x^i}_{|e^*}),\\
{\mathcal P}^A(e^*)&=&(0,\displaystyle\frac{\partial }{\partial
p_{A}}_{|e^*}),
\end{array}
\]
for $e^*\in (\tau^*)^{-1}(U).$ Here, $(x^i,p_A)$ are the local
coordinates on $E^*$ induced by the local coordinates $(x^i)$ and
the basis $\{e^A\}$ of $\Gamma(E^*)$. Using the local basis
$\{{\mathcal Y}_A,{\mathcal P}^A\}$, one may introduce, in a
natural way, local coordinates $(x^i,p_A;u^A,q_A)$ on ${\mathcal
T}^EE^*.$ If $\omega^*$ is a point of
$(\tau^{\tau^*})^{-1}((\tau^*)^{-1}(U))$, then $(x^i,p_A)$ are the
coordinates of the point $\tau^{\tau^*}(\omega^*)\in
(\tau^*)^{-1}(U)$ and
\[
\omega^*=u^A{\mathcal Y}_A(\tau^{\tau^*}(\omega^*)) + q_A {\mathcal
P}^A(\tau^{\tau^*}(\omega^*)).
\]
On the other hand, we have that
\[
\begin{array}{rcl}
\rho^{\tau^*}({\mathcal Y}_A)&=&\rho_A^i\displaystyle\frac{\partial
}{\partial x^i},\;\;\; \rho^{\tau^*}({\mathcal
P}^A)=\displaystyle\frac{\partial
}{\partial p_A},\\[8pt]
\lcf{\mathcal Y}_A,{\mathcal Y}_B\rcf^{\tau^*}&=&{\mathcal
C}_{AB}^C{\mathcal
Y}_C,\\[8pt]
\lcf{\mathcal Y}_A,{\mathcal P}^B\rcf^{\tau^*}&=&\lcf {\mathcal
P}^A, {\mathcal P}^B\rcf^{\tau^*}=0,
\end{array}
\]

\noindent for all $A$ and $B$. Thus, if $\{{\mathcal
Y}^A,{\mathcal P}_A\}$ is the dual basis of $\{{\mathcal
Y}_A,{\mathcal P}^A\}$, then
$$\begin{array}{rcl} d^{{\mathcal T}^EE^*}f&=&
\rho_A^i\displaystyle\frac{\partial f}{\partial x^i}{\mathcal
Y}^A+ \displaystyle\frac{\partial f}{\partial p_A} {\mathcal
P}_A,\\[8pt]
d^{{\mathcal T}^EE^*}{\mathcal Y}^C&=& \displaystyle -\frac{1}{2}
{\mathcal C}_{AB}^{C} {\mathcal Y}^A \wedge {\mathcal Y}^B,
\makebox[1cm]{} d^{{\mathcal T}^EE^*}{\mathcal P}_C=0,
\end{array}$$

\noindent for $f\in C^\infty(E^*).$

We may introduce a canonical section $\lambda_E$ of the vector
bundle $({\mathcal T}^EE^*)^*\to E^*$ as follows. If $e^*\in E^*$
and $(\tilde{e},v)$ is a point of the fiber of ${\mathcal T}^EE^*$
over $e^*$, then
$$\lambda_E(e^*)(\tilde{e},v)=e^*(\tilde{e}).$$
$\lambda_E$ is called \emph{the Liouville section of $({\mathcal
T}^EE^*)^*$.}

Now, \emph{the canonical symplectic section} $\Omega_E$ is the
nondegenerate closed 2-section defined by
\begin{equation}\label{sym}
\Omega_E=-d^{{\mathcal T}^EE^*} \lambda_E.
\end{equation}
In local coordinates,
$$\lambda_E(x^i,p_A)=p_A{\mathcal Y}^A ,$$
$$\Omega_E={\mathcal Y}^A\wedge {\mathcal P}_A + \frac{1}{2}
{\mathcal C}_{AB}^C p_C {\mathcal Y}^A\wedge {\mathcal Y}^B ,$$
(for more details, see \cite{LMM}).

}
\end{example}

\subsection{Lagrangian and Hamiltonian mechanics on Lie algebroids}
Given a Lagrangian function $L\in\cinfty{E}$ we define the
\emph{Cartan 1-section} $\theta_L\in\Gamma({(\T^EE)^*})$, the
\emph{Cartan 2-section} $\omega_L\in\Gamma({\wedge^2(\T^EE)^*})$
and the \emph{Lagrangian energy} $E_L\in C^{\infty}(E)$ as
$$\theta_L=S^*(d^{{\mathcal T}^EE}L) , \qquad \omega_L =
-d^{{\mathcal T}^EE}\theta_L\qquand
E_L=\mathcal{L}^{\T^EE}_{\Delta} L-L.$$
If $(x^i, y^{A})$ are local fibred coordinates on $E$,
$(\rho^i_{A}, {\mathcal C}^{C}_{AB})$ are the corresponding local
structure functions on $E$ and $\{ \X_{A}, \V_{A}\}$ is the
corresponding local basis of sections of $\T^EE$ then
\begin{equation}
\kern-3pt\label{omegaL} \omega_L\kern-3pt
=\kern-3pt\pd{^2L}{y^A\partial y^B}\X^A\wedge \V^B
\kern-2pt+\frac{1}{2}\left( \pd{^2L}{x^i\partial
y^A}\rho^i_B-\pd{^2L}{x^i\partial
y^B}\rho^i_A+\pd{L}{y^C}{\mathcal C}^C_{AB} \right)\X^A\wedge
\X^B\kern-2pt,
\end{equation}
\begin{equation}
\label{EL} E_L=\pd{L}{y^A}y^A-L.
\end{equation}

\noindent From (\ref{Lioulo}), (\ref{endverlo}), (\ref{omegaL})
and (\ref{EL}), it follows that
\begin{equation}
\label{2.4'} i_{SX} \omega_{L} = -S^*(i_{X}\omega_{L}),
\makebox[.3cm]{} i_{\Delta}\omega_{L} = -S^*(d^{{\mathcal
T}^EE}E_{L}),
\end{equation}
for $X \in \Gamma(\T^EE)$.

Now, a curve $t\to c(t)$ on $E$ is a solution of the
\emph{Euler-Lagrange equations} for $L$ if
\begin{itemize}
\item[-] $c$ is \emph{admissible} (that is, $\rho(c(t))=\dot{m}(t)$, where
$m=\tau\circ c$) and
\item[-] $\displaystyle{i_{(c(t), \dot{c}(t))}\omega_L(c(t))-d^{{\mathcal T}^EE}E_L
(c(t))=0}$, for all $t$.
\end{itemize}
If $c(t)=(x^i(t), y^A(t))$ then $c$ is a solution of the
Euler-Lagrange equations for $L$ if and only if
\begin{equation}\label{free-forces}
\dot{x}^i=\rho_A^iy^A,\;\;\;\;
\frac{d}{dt}\Bigl(\frac{\partial L}{\partial y^A}\Bigr) +
\frac{\partial L}{\partial y^C}{\mathcal C}_{AB}^C y^B -\rho_A
^i\frac{\partial L}{\partial x^i} =0.
\end{equation}

Note that if $E$ is the standard Lie algebroid $TQ$ then the above
equations are the classical Euler-Lagrange equations for $L: TQ
\to \R$.

On the other hand, the Lagrangian function $L$ is said to be
\emph{regular} if $\omega_L$ is a symplectic section, that is, if
$\omega_L$ is regular at every point as a bilinear form. In such a
case, there exists a unique solution $\xi_L$ verifying
\begin{equation}\label{symp}
i_{\xi_L}\omega_L-d^{{\mathcal T}^EE}E_L=0\; .
\end{equation}

In addition, using (\ref{2.4'}), it follows that
$i_{S\xi_{L}}\omega_{L} = i_{\Delta}\omega_{L} $ which implies
that $\xi_{L}$ is a \sode\ section. Thus, the integral curves of
$\xi_L$ (that is, the integral curves of the vector field
$\rho^{\tau}(\xi_L)$) are solutions of the Euler-Lagrange
equations for $L$. $\xi_L$ is called the \emph{Euler-Lagrange
section} associated with $L$.

>From (\ref{omegaL}), we deduce that the Lagrangian $L$ is regular
if and only if the matrix
$\displaystyle{(W_{AB})=\Big(\frac{\partial^2 L}{\partial
y^{A}\partial y^{B}}\Big)}$ is regular. Moreover, the local
expression of $\xi_L$ is
\[
\xi_L=y^A\X_A+f^A\V_A ,
\]
where the functions $f^A$ satisfy the linear equations
$$\pd{^2L}{y^B\partial y^A}f^B+\pd{^2L}{x^i\partial y^A}\rho^i_B
y^B +\pd{L}{y^C}{\mathcal C}^C_{AB}y^B -\rho^i_A\pd{L}{x^i} =0,
\mbox{ for all } A.$$

Another possibility is when the matrix
$\displaystyle{(W_{AB})=\Big(\frac{\partial^2 L}{\partial
y^{A}\partial y^{B}}\Big)}$ is non regular. This type of
lagrangians are called \emph{singular} or \emph{degenerate
lagrangians}. In such a case, $\omega_L$ is not symplectic and
Equation (\ref{symp}) has no solution, in general, and even if it
exists it will not be unique. In the next section, we will extend
the classical Gotay-Nester-Hinds algorithm \cite{GNH} for
presymplectic systems on Lie algebroids, which in particular will
be applied to the case of singular lagrangians in Section
\ref{Sec:singular}.

For an arbitrary Lagrangian function $L:E\to\R$, we introduce the
\emph{Legendre transformation} associated with $L$ as the smooth
map $leg_L:E\to E^*$ defined by
$$leg_L(e)(e')=\theta_L(e)(z'),$$
for $e,e'\in E_x$, where $z'\in{\mathcal T}^E_eE\subseteq
E_x\times T_eE$ satisfies
$$pr_1(z')=e',$$
$pr_1:{\mathcal T}^EE\to E$ being the restriction to ${\mathcal
T}^EE$ of the canonical projection $pr_1:E\times TE\to E$. The map
$leg_L:E\to E^*$ is well defined and its local expression is
\begin{equation}\label{leg-L}
leg_L(x^i,y^A)=(x^i,\displaystyle\frac{\partial L}{\partial y^A}).
\end{equation}

The Legendre transformation induces a Lie algebroid morphism
$${\mathcal T}\,leg_L:{\mathcal T}^EE\to{\mathcal T}^EE^*$$
over $leg_L:E\to E^*$ given by
$$({\mathcal T}\,leg_L)(e,v)=(e,(Tleg_L)(v)),$$
where $Tleg_L:TE\to TE^*$ is the tangent map to $leg_L:E\to E^*$.

We have that (see \cite{LMM})
\begin{equation}\label{pullbacks}
({\mathcal T}\,leg_L,leg_L)^*(\lambda_E)=\Theta_L,\;\;\;({\mathcal
T}\,leg_L,leg_L)^*(\Omega_E)=\omega_L.
\end{equation}

On the other hand, from (\ref{leg-L}), it follows that the
Lagrangian function $L$ is regular if and only if $leg_L:E\to E^*$
is a local diffeomorphism.

Next, we will assume that $L$ is \emph{hyperregular}, that is,
$leg_L:E\to E^*$ is a global diffeomorphism. Then, the pair
$({\mathcal T}\,leg_L,leg_L)$ is a Lie algebroid isomorphism.
Moreover, we may consider the \emph{Hamiltonian function}
$H:E^*\to\R$ defined by
$$H=E_L\circ leg_L^{-1}$$
and the \emph{Hamiltonian section} $\xi_H\in\Gamma({\mathcal
T}^EE^*)$ which is characterized by the condition
$$i_{\xi_H}\Omega_E=d^{{\mathcal T}^EE^*}H.$$

The integral curves of the vector field $\rho^{\tau^*}(\xi_H)$ on
$E^*$ satisfy the \emph{Hamilton equations for $H$}
$$\displaystyle\frac{dx^i}{dt}=\rho_A^i\frac{\partial H}{\partial
p_A},\;\;\;\frac{d p_A}{dt}=-\Big(\rho_A^i\frac{\partial
H}{\partial x^i}+{\mathcal C}_{AB}^Cp_C\frac{\partial H}{\partial
p_B}\Big)$$ for $i\in\{1,\dots,m\}$ and $A\in\{1,\dots,n\}$ (see
\cite{LMM}).

In addition, the Euler-Lagrange section $\xi_L$ associated with
$L$ and the Hamil\-to\-nian section $\xi_H$ are $({\mathcal
T}\,leg_L,leg_L)$-related, that is,
$$\xi_H\circ leg_L={\mathcal T}\,leg_L\circ\xi_L.$$

Thus, if $\gamma:I\to E$ is a solution of the Euler-Lagrange
equations associated with $L$, then $\mu=leg_L\circ\gamma:I\to
E^*$ is a solution of the Hamilton equations for $H$ and,
conversely, if $\mu:I\to E^*$ is a solution of the Hamilton
equations for $H$ then $\gamma=leg_L^{-1}\circ\mu$ is a solution
of the Euler-Lagrange equations for $L$ (for more details, see
\cite{LMM}).

\section{Constraint algorithm and reduction for presymplectic
Lie algebroids}\label{pre}

\subsection{Constraint algorithm for presymplectic Lie
algebroids}\label{sec3.1}

Let $\tau :E\to Q$ be a Lie algebroid and suppose that $\Omega \in
\Gamma(\wedge ^2E^*)$. Then, we can define the vector bundle
morphism $\flat_{\Omega}:E\to E^*$ (over the identity of $Q$) as
follows
$$\flat_{\Omega}(e)=i(e)\Omega(x),\makebox[1cm]{for}e\in E_x.$$

Now, if $x\in Q$ and $F_x$ is a subspace of $E_x$, we may
introduce the vector subspace $F_x^\perp$ of $E_x$ given by
$$F_x^\perp=\{e\in E_x \,|\, \Omega(x)(e,f)=0,\forall f\in F_x\}.$$

Then, using a well-known result (see, for instance, \cite{LeRo}),
we have that
\begin{equation}\label{Forortho}
dim F_{x}^\perp = dim E_{x} - dim F_{x} + dim (E_{x}^\perp \cap
F_{x}).
\end{equation}

On the other hand, if $\flat_{\Omega_x}={\flat_\Omega}_{|E_x}$ it
is easy to prove that
\begin{equation}\label{Subset}
\flat_{\Omega_{x}}(F_{x}) \subseteq (F_{x}^\perp)^{0},
\end{equation}
where $(F_{x}^\perp)^{0}$ is the annihilator of the subspace
$F_{x}^\perp$. Moreover, using (\ref{Forortho}), we obtain that
\[
dim (F_{x}^\perp)^{0} = dim F_{x} - dim (E_{x}^\perp \cap F_{x}) =
dim (\flat_{\Omega_{x}}(F_{x})).
\]
Thus, from (\ref{Subset}), we deduce that
\begin{equation}\label{flat1}
\flat_{\Omega_x}(F_x)=(F_x^\perp)^\circ.
\end{equation}

Next, we will assume that $\Omega$ is a presymplectic 2-section
($d^E\Omega =0$) and that $\alpha \in \Gamma(E^*)$ is a closed
1-section ($d^E\alpha =0$). Furthermore, we will assume that the
kernel of $\Omega$ is a vector subbundle of $E$.

The dynamics of the presymplectic system defined by $(\Omega,
\alpha)$ is given by a section $X\in \Gamma(E)$ satisfying the
dynamical equation
\begin{equation}\label{presym}
i_X\Omega =\alpha\; .
\end{equation}
In general, a section $X$ satisfying (\ref{presym}) cannot be found
in all points of $E$. First, we look for the points where
(\ref{presym}) has sense. We define
\[
Q_1  =\{ x\in Q \,|\, \exists e\in E_x:\; i(e)\Omega(x)=\alpha
(x)\}=\{  x\in Q \,|\, \alpha(x)\in\flat_{\Omega_x}(E_x)\}.
\]
>From (\ref{flat1}), it follows that
\begin{equation}\label{q1}
Q_1  =\{ x\in Q \,|\, \alpha (x)(e)=0,\makebox[1.5cm]{for all}e\in
Ker\Omega(x)=E_x^\perp \}.
\end{equation}

If $Q_1$ is an embedded submanifold of $Q$, then we deduce that
there exists $X:Q_1\to E$ a section of $\tau:E\to Q$ along $Q_1$
such that (\ref{presym}) holds. But $\rho (X)$ is not, in general,
tangent to $Q_1$. Thus, we have that to restrict to $E_1=\rho
^{-1} (TQ_1)$. We remark that, provided that $E_1$ is a manifold
and $\tau_1=\tau_{|E_1}:E_1\to Q_1$ is a vector bundle,
$\tau_1:E_1\to Q_1$ is a Lie subalgebroid of $E\to Q$.

Now, we must consider the subset $Q_2$ of $Q_1$ defined by
$$\begin{array}{cl}
Q_2&=\{ x\in Q_1 \,|\, \alpha
(x)\in\flat_{\Omega_x}((E_1)_x)=\flat_{\Omega_x}(\rho^{-1}(T_xQ_1))\}\\[5pt]
&=\{ x\in Q_1 \,|\, \alpha (x)(e)=0,\, \mbox{ for all } e\in
(E_1)_x^\perp=(\rho ^{-1} (T_xQ_1))^\perp \}. \end{array}$$

If $Q_2$ is an embedded submanifold of $Q_1$, then we deduce that
there exists $X:Q_2\to E_1$ a section of $\tau_1:E_1\to Q_1$ along
$Q_2$ such that (\ref{presym}) holds. However, $\rho(X)$ is not,
in general, tangent to $Q_2$. Therefore, we have that to restrict
to $E_2=\rho^{-1}(TQ_2)$. As above, if $\tau_2=\tau_{|E_2}:E_2\to
Q_2$ is a vector bundle, it follows that $\tau_2:E_2\to Q_2$ is a
Lie subalgebroid of $\tau_1:E_1\to Q_1$.

Consequently, if we repeat the process, we obtain a sequence of
Lie subalgebroids (by assumption):

\begin{picture}(500,70)(0,0)
\put(65,20){\makebox(0,0){$\dots$}} \put(72,20){$\hookrightarrow$}
\put(97,20){\makebox(0,5){$E_{k+1}$}}
\put(112,20){$\hookrightarrow$}
\put(132,20){\makebox(0,5){$E_{k}$}}
\put(142,20){$\hookrightarrow$}
\put(162,20){\makebox(0,0){$\dots$}}
\put(170,20){$\hookrightarrow$}
\put(190,20){\makebox(0,5){$E_{2}$}}
\put(200,20){$\hookrightarrow$}
\put(220,20){\makebox(0,5){$E_{1}$}}
\put(230,20){$\hookrightarrow$}
\put(260,20){\makebox(0,5){$E_{0}=E$}}
\put(94,37){$\tau_{k+1}$}\put(90,31){\vector(0,1){20}}
\put(133,37){$\tau_{k}$}\put(130,31){\vector(0,1){20}}
\put(191,37){$\tau_{2}$} \put(188,31){\vector(0,1){20}}
\put(221,37){$\tau_{1}$} \put(218,31){\vector(0,1){20}}
\put(252,37){$\tau_0=\tau$} \put(248,31){\vector(0,1){20}}
\put(65,55){\makebox(0,0){$\dots$}} \put(72,55){$\hookrightarrow$}
\put(97,55){\makebox(0,5){$Q_{k+1}$}}
\put(112,55){$\hookrightarrow$}
\put(132,55){\makebox(0,5){$Q_{k}$}}
\put(142,55){$\hookrightarrow$}
\put(162,55){\makebox(0,0){$\dots$}}
\put(170,55){$\hookrightarrow$}
\put(190,55){\makebox(0,5){$Q_{2}$}}
\put(200,55){$\hookrightarrow$}
\put(220,55){\makebox(0,5){$Q_{1}$}}
\put(230,55){$\hookrightarrow$}
\put(260,55){\makebox(0,5){$Q_{0}=Q$}}

 \end{picture}

\vspace{-0.5cm}\noindent where
\begin{equation}\label{qkmas1}
Q_{k+1}=\{ x\in Q_k \,|\, \alpha(x)(e)=0,\mbox{ for all }
e\in(\rho^{-1}(T_xQ_k))^\perp\} \end{equation} and
$$E_{k+1}=\rho^{-1}(TQ_{k+1}).$$

If there exists $k\in\N$ such that $Q_k=Q_{k+1}$, then we say that
the sequence stabilizes. In such a case, there exists a
well-defined (but non necessarily unique) dynamics on the final
constraint submanifold $Q_f=Q_k$. We write

\[
{Q}_f=Q_{k+1}=Q_k,\qquad E_f=E_{k+1}=E_k=\rho ^{-1} (TQ_k).
\]

\noindent Then, $\tau_f=\tau_k:E_f=E_k\to Q_f=Q_k$ is a Lie
subalgebroid of $\tau:E\longrightarrow Q$ (the Lie algebroid
restriction of $E$ to $E_f$). From the construction of the
constraint algorithm, we deduce that there exists a section $X\in
\Gamma(E_f)$, verifying (\ref{presym}). Moreover, if $X\in
\Gamma(E_f)$ is a solution of the equation (\ref{presym}), then
every arbitrary solution is of the form $X'=X+Y$, where
$Y\in\Gamma(E_f)$ and $Y(x)\in\ker \Omega(x)$, for all $x\in Q_f$.
In addition, if we denote by $\Omega_f$ and $\alpha_f$ the
restriction of $\Omega$ and $\alpha$, respectively, to the Lie
algebroid $E_f\longrightarrow Q_f$, we have that $\Omega_f$ is a
presymplectic 2-section and then any $X\in \Gamma(E_f)$ verifying
Equation (\ref{presym}) also satisfies
\begin{equation}\label{presym2}
i_X\Omega_f=\alpha_f
\end{equation}
but, in principle, there are solutions of (\ref{presym2})
which are not solutions of (\ref{presym}) since $\ker
\Omega\cap E_f \subset \ker\Omega_f$.

\begin{remark}
Note that one can generalize the previous procedure to the general
setting of \textit{implicit differential equations on a Lie
algebroid}. More precisely, let $\tau :E\to Q$ be a Lie algebroid
and $S\subset E$ be a submanifold of $E$ (not necessarily a vector
subbundle). Then, the corresponding sequence of submanifolds of
$E$ is
\[
\begin{array}{cl}
S_0&=S\\[5pt]
S_1&=S_0\cap \rho ^{-1} \big (T \tau (S_0)\big )\\
\vdots & \\
\\
S_{k+1}&=S_k\cap \rho ^{-1}\big (T \tau (S_k)\big )
\\
\vdots &
\end{array}
\]
In our case, $S_k=\rho ^{-1} (TQ_k)$ (equivalently, $Q_k=\tau
(S_k)$).
\end{remark}

\subsection{Reduction of presymplectic Lie
algebroids}\label{sec3.2}

Let  $(E,\lcf\cdot,\cdot \rcf,\rho)$ and
$(E'\kern-2pt,\lcf\cdot,\cdot\rcf'\kern-2pt,$ $\rho')$ be two Lie
algebroids over $Q$ and $Q'$, respectively. Suppose that
$\Omega\in\Gamma(\wedge^2E^*)$ (respectively,
$\Omega'\in\Gamma(\wedge^2(E')^*)$) is a presymplectic 2-section
on $\tau:E\to Q$ (respectively, $\tau':E'\to Q'$) and that
$\alpha\in\Gamma(E^*)$ (respectively, $\alpha'\in\Gamma((E')^*)$)
is a closed 1-section on $\tau:E\to Q$ (respectively, $\tau':E'\to
Q'$). Then, we may consider the corresponding dynamical equations
$$i_X\Omega=\alpha,\;\;X\in\Gamma(E),$$
$$i_{X'}\Omega'=\alpha',\;\;X'\in\Gamma(E').$$

If we apply our constraint algorithm to the first problem, we will
obtain a sequence of Lie subalgebroids of $\tau:E\to Q$

\begin{picture}(500,70)(0,0)
\put(65,20){\makebox(0,0){$\dots$}} \put(72,20){$\hookrightarrow$}
\put(97,20){\makebox(0,5){$E_{k+1}$}}
\put(112,20){$\hookrightarrow$}
\put(132,20){\makebox(0,5){$E_{k}$}}
\put(142,20){$\hookrightarrow$}
\put(162,20){\makebox(0,0){$\dots$}}
\put(170,20){$\hookrightarrow$}
\put(190,20){\makebox(0,5){$E_{2}$}}
\put(200,20){$\hookrightarrow$}
\put(220,20){\makebox(0,5){$E_{1}$}}
\put(230,20){$\hookrightarrow$}
\put(260,20){\makebox(0,5){$E_{0}=E$}}
\put(94,37){$\tau_{k+1}$}\put(90,31){\vector(0,1){20}}
\put(133,37){$\tau_{k}$}\put(130,31){\vector(0,1){20}}
\put(191,37){$\tau_{2}$} \put(188,31){\vector(0,1){20}}
\put(221,37){$\tau_{1}$} \put(218,31){\vector(0,1){20}}
\put(252,37){$\tau_0=\tau$} \put(248,31){\vector(0,1){20}}
\put(65,55){\makebox(0,0){$\dots$}} \put(72,55){$\hookrightarrow$}
\put(97,55){\makebox(0,5){$Q_{k+1}$}}
\put(112,55){$\hookrightarrow$}
\put(132,55){\makebox(0,5){$Q_{k}$}}
\put(142,55){$\hookrightarrow$}
\put(162,55){\makebox(0,0){$\dots$}}
\put(170,55){$\hookrightarrow$}
\put(190,55){\makebox(0,5){$Q_{2}$}}
\put(200,55){$\hookrightarrow$}
\put(220,55){\makebox(0,5){$Q_{1}$}}
\put(230,55){$\hookrightarrow$}
\put(260,55){\makebox(0,5){$Q_{0}=Q$}}
 \end{picture}

\vspace{-0.5cm}In a similar way, if we apply our constraint
algorithm to the second problem, we will obtain a sequence of Lie
subalgebroids of $\tau':E'\to Q'$

\begin{picture}(500,70)(0,0)
\put(65,20){\makebox(0,0){$\dots$}} \put(72,20){$\hookrightarrow$}
\put(97,20){\makebox(0,5){$E_{k+1}'$}}
\put(112,20){$\hookrightarrow$}
\put(132,20){\makebox(0,5){$E_{k}'$}}
\put(142,20){$\hookrightarrow$}
\put(162,20){\makebox(0,0){$\dots$}}
\put(170,20){$\hookrightarrow$}
\put(190,20){\makebox(0,5){$E_{2}'$}}
\put(200,20){$\hookrightarrow$}
\put(220,20){\makebox(0,5){$E_{1}'$}}
\put(230,20){$\hookrightarrow$}
\put(260,20){\makebox(0,5){$E_{0}'=E'$}}
\put(94,37){$\tau_{k+1}'$}\put(90,31){\vector(0,1){20}}
\put(133,37){$\tau_{k}'$}\put(130,31){\vector(0,1){20}}
\put(191,37){$\tau_{2}'$} \put(188,31){\vector(0,1){20}}
\put(221,37){$\tau_{1}'$} \put(218,31){\vector(0,1){20}}
\put(252,37){$\tau_0'=\tau'$} \put(248,31){\vector(0,1){20}}
\put(65,55){\makebox(0,0){$\dots$}} \put(72,55){$\hookrightarrow$}
\put(97,55){\makebox(0,5){$Q'_{k+1}$}}
\put(112,55){$\hookrightarrow$}
\put(132,55){\makebox(0,5){$Q'_{k}$}}
\put(142,55){$\hookrightarrow$}
\put(162,55){\makebox(0,0){$\dots$}}
\put(170,55){$\hookrightarrow$}
\put(190,55){\makebox(0,5){$Q'_{2}$}}
\put(200,55){$\hookrightarrow$}
\put(220,55){\makebox(0,5){$Q'_{1}$}}
\put(230,55){$\hookrightarrow$}
\put(260,55){\makebox(0,5){$Q'_{0}=Q'$}}
 \end{picture}

\vspace{-0.5cm}On the other hand, it is clear that the restriction
$\Omega_k$ (respectively, $\Omega_k'$) of $\Omega$ (respectively,
$\Omega'$) to the Lie subalgebroid $\tau_k:E_k\to Q_k$
(respectively, $\tau'_k:E'_k\to Q'_k$) is a presymplectic section
of $\tau_k:E_k\to Q_k$ (respectively, $\tau'_k:E'_k\to Q'_k$).
Moreover, if $\alpha_k$ (respectively, $\alpha_k'$) is the
restriction of $\alpha$ (respectively, $\alpha'$) to
$\tau_k:E_k\to Q_k$ (respectively, $\tau'_k:E'_k\to Q'_k$), we may
consider the dynamical problem
$$i_{X_k}\Omega_k=\alpha_k,\;\;X_k\in\Gamma(E_k),$$
(respectively,
$i_{X'_k}\Omega'_k=\alpha'_k,\;\;X'_k\in\Gamma(E'_k)$) on
$\tau_k:E_k\to Q_k$ (respectively, $\tau'_k:E'_k\to Q'_k$).

Now, suppose that the pair $(\Pi,\pi)$

\begin{picture}(375,90)(80,20)
\put(195,20){\makebox(0,0){$Q$}}
\put(250,25){$\pi$}\put(210,20){\vector(1,0){80}}
\put(310,20){\makebox(0,0){$Q'$}} \put(185,50){$\tau$}
\put(195,70){\vector(0,-1){40}} \put(320,50){$\tau'$}
\put(310,70){\vector(0,-1){40}} \put(195,80){\makebox(0,0){$E$}}
\put(250,85){$\Pi$}\put(210,80){\vector(1,0){80}}
\put(310,80){\makebox(0,0){$E'$}} \end{picture}

\vspace{0.4cm}\noindent is a \emph{dynamical Lie algebroid
epimorphism} between $E$ and $E'$. This means that:
\begin{enumerate}
\item The pair $(\Pi,\pi)$ is a Lie algebroid morphism,
\item $\pi:Q\to Q'$ is a surjective submersion and
$\Pi_{|E_x}:E_x\to E'_{\pi(x)}$ is a linear epimorphism, for all
$x\in Q$, and
\item $(\Pi,\pi)^*\Omega'=\Omega$ and $(\Pi,\pi)^*\alpha'=\alpha$.
\end{enumerate}

Then, we will see that the Lie subalgebroids in the two above
sequences are related by dynamical Lie algebroid epimorphisms.
First, we will prove the result for $k = 1$.

\begin{lemma}\label{lemma} If $(\Pi,\pi)$ is a dynamical Lie algebroid
epimorphism, we have that:
\begin{enumerate}
\item[a)] $\pi(Q_1)=Q_1'$ and $\Pi(E_1)=E_1'$.
\item[b)] If $x_1\in Q_1$, then $\pi^{-1}(\pi(x_1))\subseteq Q_1$
and $Ker(\Pi_{|E_{x_1}})\subseteq (E_1)_{x_1}$.
\item[c)] If $\pi_1:Q_1\to Q_1'$ and $\Pi_1:E_1\to E_1'$ are the
restrictions to $Q_1$ and $E_1$ of $\pi:Q\to Q'$ and $\Pi:E\to
E'$, respectively, then the pair $(\Pi_1,\pi_1)$ is a dynamical
Lie algebroid epimorphism.
\end{enumerate}
\end{lemma}
\begin{proof} If $x\in Q$ then, using that
$(\Pi,\pi)^*\Omega'=\Omega$ and the fact that $\Pi_{|E_x}:E_x\to
E'_{\pi(x)}$ is a linear epimorphism, we deduce that
\begin{equation}\label{ker}
\Pi(Ker\Omega(x))=Ker\Omega'(\pi(x)).
\end{equation}

Thus, from (\ref{q1}), (\ref{ker}) and since
$(\Pi,\pi)^*\alpha'=\alpha$, it follows that
$$\pi(Q_1)\subseteq Q_1'.$$

Conversely, if $x_1'\in Q_1'$ and $x\in \pi^{-1}(x_1')$ then,
using again (\ref{q1}), (\ref{ker}) and the fact that
$(\Pi,\pi)^*\alpha'=\alpha$, we obtain that $x\in Q_1$. This
proves that
$$Q_1'\subseteq\pi(Q_1)$$
and the following result
\begin{equation}\label{fiber}
x_1\in Q_1\Rightarrow\pi^{-1}(\pi(x_1))\subseteq Q_1.
\end{equation}

Now, we will see that $\pi_1:Q_1\to Q_1'$ is a submersion.

In fact, if $x_1\in Q_1$ and $x_1'=\pi_1(x_1)$ then there exist an
open subset $U'$ of $Q'$, $x_1'\in U'$, and an smooth local
section $s':U'\to Q$ of the submersion $\pi:Q\to Q'$ such that
$s'(x_1')=x_1$. Note that, using (\ref{fiber}), we conclude that
the restriction $s_1'$ of $s'$ to the open subset $U_1'=U'\cap
Q_1'$ of $Q_1'$ takes values in $Q_1$. Therefore,
$s_1':U_1'\subseteq Q_1'\to Q_1$ is a smooth map, $s_1'(x_1')=x_1$
and $s_1'\circ\pi_1=Id$. Consequently, $\pi_1:Q_1\to Q_1'$ is a
submersion.

Next, we will prove that
$$\Pi((E_1)_{x_1})=\Pi(\rho^{-1}(T_{x_1}Q_1))=(\rho')^{-1}(T_{\pi(x_1)}Q_1')=(E_1')_{\pi(x_1)},\mbox{ for }x_1\in
Q_1.$$

Since
\begin{equation}\label{rhoE1}
(\rho'\circ\Pi)_{|E_{x_1}}=(T\pi\circ\rho)_{|E_{x_1}},
\end{equation}
it follows that
$$\Pi((E_1)_{x_1})\subseteq(E_1')_{\pi(x_1)}.$$

Conversely, suppose that $e_1'\in
(E_1')_{\pi(x_1)}=(\rho')^{-1}(T_{\pi(x_1)}Q_1')$. Then, we can
choose $e\in E_{x_1}$ such that $\Pi(e)=e_1'$. Thus, from
(\ref{rhoE1}), we have that
$$(T\pi)(\rho(e))\in T_{\pi(x_1)}Q_1'.$$

Now, using that $\pi_1:Q_1\to Q_1'$ is a submersion, we deduce
that there exists $v_1\in T_{x_1}Q_1$ such that
$$(T\pi)(v_1)=(T\pi)(\rho(e)),$$
that is,
$$v_1-\rho(e)\in T_{x_1}(\pi^{-1}(\pi(x_1)))\subseteq
T_{x_1}Q_1.$$

Therefore, $\rho(e)\in T_{x_1}Q_1$ and
$$e\in \rho^{-1}(T_{x_1}Q_1)=(E_1)_{x_1}.$$

On the other hand, if $e\in Ker(\Pi_{|E_{x_1}})$, then
$$0=\Pi(e)\in(E_1')_{\pi(x_1)},$$
and, proceeding as above, we conclude that $e\in(E_1)_{x_1}$.

Finally, using that the pair $(\Pi,\pi)$ is a dynamical Lie
algebroid epimorphism, we obtain that the pair $(\Pi_1,\pi_1)$ is
also a dynamical Lie algebroid epimorphism.

\end{proof}

Next, we will prove the following theorem.

\begin{theorem}\label{Th} Let $(\Pi,\pi)$ be a dynamical Lie algebroid
epimorphism between $E$ and $E'$. Then, we have that:
\begin{enumerate}
\item $\pi(Q_k)=Q_k'$ and $\Pi(E_k)=E_k'$, for all $k$.
\item If $x_k\in Q_k$, then $\pi^{-1}(\pi(x_k))\subseteq Q_k$ and
$Ker(\Pi_{|E_{x_k}})\subseteq (E_k)_{x_k}$, for all $k$.
\item If $\pi_k:Q_k\to Q_k'$ and $\Pi_k:E_k\to E_k'$ are the
restrictions to $Q_k$ and $E_k$ of $\pi:Q\to Q'$ and $\Pi:E\to
E'$, respectively, then the pair $(\Pi_k,\pi_k)$ is a dynamical
Lie algebroid epimorphism, for all $k$.
\end{enumerate}
\end{theorem}

\begin{proof} The result holds for $k=0,1$. Then, we will proceed
by induction.

Assume that the result holds for $k\in\{0,1,\dots,N\}$. Then, we
will prove it for $k=N+1$.

Note that if $k\in\{0,1,\dots,N\}$ and $x_k\in Q_k$ then, using
the following facts
$$(\Pi,\pi)^*\Omega'=\Omega,\;\;\Pi((E_k)_{x_k})=(E_k')_{\pi(x_k)}\mbox{ and }\Pi(E_{x_k})=E_{\pi(x_k)}',$$
we obtain that
\begin{equation}\label{ortho}
\Pi((E_k)_{x_k}^\perp)=(E_k')_{\pi(x_k)}^\perp. \end{equation}
Thus, proceeding as in the proof of Lemma \ref{lemma}, we deduce
the result.

\end{proof}

We remark that the behavior of the two constraint algorithms is
the same. In fact, if we obtain a final Lie subalgebroid
$\tau_f=\tau_k:E_f=E_k\to Q_f=Q_k$ for the first problem (that is,
if $Q_k=Q_{k+1}$) then, from Theorem \ref{Th}, it follows that
$Q_k'=Q_{k+1}'$ and we have a final Lie subalgebroid
$\tau_f'=\tau'_k:E'_f=E'_k\to Q'_f=Q'_k$ for the second problem.
Conversely, if the second constraint algorithm stops at a certain
$k$ (that is,  $Q_k'=Q_{k+1}'$) then, using (\ref{qkmas1}) and
(\ref{ortho}), we deduce that $Q_k=Q_{k+1}$, i.e., the first
constraint algorithm also stops at the level $k$.

Now, suppose that $X:Q_k\to E_k$ is a section of the Lie algebroid
$\tau_k:E_k\to Q_k$ such that
$$i_X\Omega_{|Q_k}=\alpha_{|Q_k}$$
and $X$ is $(\Pi_k,\pi_k)$-projectable, i.e., there exists
$X'\in\Gamma(E_k')$ satisfying
$$X'\circ\pi_k=\Pi_k\circ X.$$

Then, using that $(\Pi,\pi)^*\Omega'=\Omega$ and that
$(\Pi,\pi)^*\alpha'=\alpha$, we obtain that
$$i_{X'}\Omega'_{|Q_k'}=\alpha'_{|Q_k'}.$$

In others words, $X'$ is a solution (along $Q_k'$) of the second
dynamical problem.

Conversely, if $X'\in\Gamma(E_k')$ is a solution of the dynamical
equation
$$i_{X'}\Omega'_{|Q_k'}=\alpha'_{|Q_k'},$$
then we can choose $X\in\Gamma(E_k)$ such that
$$X'\circ\pi_k=\Pi_k\circ X$$
and, since $(\Pi,\pi)^*\Omega'=\Omega$ and
$(\Pi,\pi)^*\alpha'=\alpha$, we conclude that
$$i_X\Omega_{|Q_k}=\alpha_{|Q_k}.$$

\section{Singular Lagrangian systems on Lie algebroids}\label{Sec:singular}
Let $L:E\to\R$ be a Lagrangian function on a Lie algebroid
$\tau:E\to Q$.

Denote by $\omega_L$ and $E_L$ the Cartan 2-section and the
Lagrangian energy, respectively, associated with $L$. Then,
$\omega_L$ is not, in general, a symplectic section and, thus, the
dynamical equation
$$i_X\omega_L=d^{{\mathcal T}^EE}E_L$$
has not, in general, solution. Moreover, if there exists a
solution of the above equation, it is not, in general, a second
order differential equation and it is not, in general, unique.

Note that the Legendre transformation $leg_L:E\to E^*$ associated
with $L$ is not, in general, a local diffeomorphism.

\begin{definition} The Lagrangian function $L$ is said to be
\emph{almost regular} if the following conditions hold:
\begin{itemize}
\item[i)] The subset $M_1=leg_L(E)$ of $E^*$ is an embedded
submanifold of $E^*$.
\item[ii)] The map $leg_1:E\to M_1$ induced by the Legendre
transformation is a submersion with connected fibres.
\end{itemize}\end{definition}

In what follows, we will assume that $L$ is an almost regular
Lagrangian.

Then, we may prove the following result.

\begin{proposition}\label{prop2}
The Lagrangian energy $E_L$ is a basic function with respect to
the submersion $leg_1:E\to M_1$, that is, there exists a
Hamiltonian function $H$ on $M_1$ such that
$$H\circ leg_1=E_L.$$
\end{proposition}
\begin{proof}
Suppose that $e$ is a point of $E$ and that $(x^i,y^A)$ are fibred
local coordinates in an open subset of $E$ which contains to $e$.
Then, using (\ref{leg-L}), we deduce that
$X=\lambda^A\frac{\partial}{\partial y^A}_{|e}\in T_eE$ is
vertical with respect to the submersion $leg_1:E\to M_1$ if and
only if
$$\lambda^A\displaystyle\frac{\partial^2L}{\partial y^A\partial
y^B}_{|e}=0,\mbox{ for all }B.$$ Thus, if $X$ is vertical, from
(\ref{EL}), it follows that
$$X(E_L)=\lambda^A\displaystyle\frac{\partial L}{\partial
y^A}_{|e}+\lambda^A y^B(e)\frac{\partial^2L}{\partial y^A\partial
y^B}_{|e}-\lambda^A\frac{\partial L}{\partial y^A}_{|e}=0.$$ This
ends the proof of the result.
\end{proof}

Now, since $\tau^*_{|M_1}:M_1\to Q$ is a fibration, one may
consider the prolongation ${\mathcal T}^EM_1$ of the Lie algebroid
$\tau:E\to Q$ over $\tau^*_{|M_1}$ or, in other words, the
$E$-tangent bundle to $M_1$.

Then, the canonical symplectic section $\Omega_E$ on ${\mathcal
T}^EE^*\to E^*$ induces a presymplectic section $\Omega_1$ on the
Lie algebroid ${\mathcal T}^EM_1\to M_1$. Moreover, the submersion
$leg_1:E\to M_1$ induces, in a natural way, a Lie algebroid
epimorphism
$${\mathcal T}\,leg_1:{\mathcal T}^EE\to{\mathcal T}^EM_1$$
(over $leg_1$). In addition, using (\ref{pullbacks}) and
Proposition \ref{prop2}, we obtain that
$$({\mathcal T}\,leg_1,leg_1)^*(\Omega_1)=\omega_L,\;\;\;({\mathcal
T}\,leg_1,leg_1)^*(d^{{\mathcal T}^EE}E_L)=d^{{\mathcal
T}^EM_1}H.$$ Therefore, we have proved the following result.

\begin{proposition}\label{prop3}
The pair $({\mathcal T}\,leg_1,leg_1)$ is a dynamical Lie
algebroid epimorphism between the presymplectic Lie algebroids
$({\mathcal T}^EE,\omega_L,d^{{\mathcal T}^EE}E_L)$ and
$({\mathcal T}^EM_1,\Omega_1,$ $d^{{\mathcal T}^EM_1}H)$.
\end{proposition}

The following diagram illustrates the above situation.

\begin{picture}(130,150)(-100,0)
\put(5,75){\makebox(0,0){$E$}}
\put(50,78){$leg_1$}\put(15,75){\vector(1,0){90}}
\put(118,75){\makebox(0,0){$M_1$}}
\put(118,125){\vector(0,-1){40}} \put(5,125){\vector(0,-1){40}}
\put(5,135){\makebox(0,0){${\mathcal T}^EE$}}
\put(48,138){${\mathcal T}\,leg_1$}\put(20,135){\vector(1,0){85}}
\put(123,135){\makebox(0,0){${\mathcal T}^EM_1$}}
\put(7,65){\vector(1,-1){40}} \put(112,65){\vector(-1,-1){40}}
\put(60,20){\makebox(0,0){$Q$}} \put(20,42){\makebox(0,0){$\tau$}}
\put(105,42){\makebox(0,0){$\tau^*_{|M_1}$}}

\end{picture}

\vspace{-0.35cm}Now, we consider the following dynamical equations
\begin{equation}\label{firsteq}
i_X\omega_L=d^{{\mathcal T}^EE}E_L,\;\;\mbox{ with
}X\in\Gamma({\mathcal T}^EE)
\end{equation}
and
\begin{equation}\label{secondeq}
i_Y\Omega_1=d^{{\mathcal T}^EM_1}H,\;\;\mbox{ with
}Y\in\Gamma({\mathcal T}^EM_1).
\end{equation}

In general, a section $X\in\Gamma({\mathcal T}^EE)$ (respectively,
$Y\in\Gamma({\mathcal T}^EM_1)$) satisfying (\ref{firsteq})
(respectively, (\ref{secondeq})) can not be found in all the
points of $E$ (respectively, $M_1$). Thus, we must apply the
general constraint algorithm developed in Section \ref{sec3.1} for
an arbitrary presymplectic system.

Assume that this algorithm stops at the level $k$ for the first
dynamical equation, that is, there exists a Lie subalgebroid
$({\mathcal T}^EE)_k$ of ${\mathcal T}^EE$ over a submanifold
$E_k$ of $E$ and a section $X_k\in\Gamma(({\mathcal T}^EE)_k)$
such that
$$(i_{X_k}\omega_L)_{|E_k}=(d^{{\mathcal T}^EE}E_L)_{|E_k}.$$

Note that $({\mathcal T}^EE)_k=(\rho^\tau)^{-1}(TE_k)$, where
$\rho^\tau:{\mathcal T}^EE\to TE$ is the anchor map of the Lie
algebroid ${\mathcal T}^EE\to E$. Moreover, using Proposition
\ref{prop3} and the results of Section \ref{sec3.2}, we deduce
that the constraint algorithm also stops at the level $k$ for the
second equation. In fact, we have that:

\begin{enumerate}
\item $leg_1(E_k)=M_{k+1}$ is a submanifold of $M_1$ and
$$({\mathcal T}^EM_1)_{k}=({\mathcal T}\,leg_1)(({\mathcal
T}^EE)_k)=(\rho^{\tau^*})^{-1}(TM_{k+1})$$ is a Lie algebroid over
$M_{k+1}$, $\rho^{\tau^*}:{\mathcal T}^EE^*\to TE^*$ being the
anchor map of the Lie algebroid ${\mathcal T}^EE^*\to E^*$.
\item If $e_k\in E_k$ then $leg_1^{-1}(leg_1(e_k))\subseteq
E_k$ and $Ker({\mathcal T}_{e_k}leg_1)\subseteq({\mathcal
T}^E_{e_k}E)_k$. Note that, from (\ref{omegaL}) and (\ref{leg-L}),
it follows that
\begin{equation}\label{nucleodeleg}
Ker({\mathcal T}_{e_k}leg_1)=Ker\omega_L(e_k)\cap({\mathcal
T}^E_{e_k}E)^V.
\end{equation}
\item If $leg_{k+1}:E_k\to M_{k+1}$ and
$Leg_{k+1}:({\mathcal T}^EE)_k\to({\mathcal T}^EM_1)_{k}$ are the
restrictions to $E_k$ and $({\mathcal T}^EE)_k$ of $leg_1:E=E_0\to
M_1$ and ${\mathcal T}\,leg_1:{\mathcal T}^EE\to{\mathcal
T}^EM_1$, respectively, then the pair $(Leg_{k+1},leg_{k+1})$ is a
dynamical Lie algebroid epimorphism.
\item If $X_k\in\Gamma(({\mathcal T}^EE)_k)$ is such that
$(i_{X_k}\omega_L)_{|E_k}=(d^{{\mathcal T}^EE}E_L)_{|E_k}$ and
$X_k$ is $(Leg_{k+1},leg_{k+1})$-projectable, i.e., there exists
$Y_{k}\in$ $\Gamma(({\mathcal T}^EM_1)_{k})$ sa\-tis\-fying
\[
Y_{k}\circ leg_{k+1}=Leg_{k+1}\circ X_k,
\]
then
\[
(i_{Y_{k}}\Omega_1)_{|M_{k+1}}=(d^{{\mathcal
T}^EM_1}H)_{|M_{k+1}}.
\]
\item If $Y_{k}\in\Gamma(({\mathcal T}^EM_1)_{k})$ is a
solution of the dynamical equation
\[
(i_{Y_{k}}\Omega_1)_{|M_{k+1}}=(d^{{\mathcal T}^EM_1}H)_{|M_{k+1}}
\]
then we can choose $X_k\in\Gamma(({\mathcal T}^EE)_k)$ such that
\[
Y_{k}\circ leg_{k+1}=Leg_{k+1}\circ X_k\mbox{ and
}(i_{X_k}\omega_L)_{|E_k}=(d^{{\mathcal T}^EE}E_L)_{|E_k}.
\]
\end{enumerate}

Now, suppose that $X\in\Gamma(({\mathcal T}^EE)_k)$ is a solution
of the dynamical equation
$$(i_{X}\omega_L)_{|E_k}=(d^{{\mathcal T}^EE}E_L)_{|E_k}$$
and that $X$ is $(Leg_{k+1},leg_{k+1})$-projectable over
$Y\in\Gamma(({\mathcal T}^EM_1)_{k})$.

Then,
\begin{equation}\label{HamilY}
(i_{Y}\Omega_1)_{|M_{k+1}}=(d^{{\mathcal T}^EM_1}H)_{|M_{k+1}}.
\end{equation}

The following diagram illustrates the above situation.

\begin{picture}(375,150)(-50,0)
\put(5,5){\makebox(0,0){${\mathcal T}^EM_1$}}
\put(22,5){\vector(1,0){200}} \put(235,5){\makebox(0,0){$M_1$}}
\put(5,125){\vector(0,-1){110}}
\put(-30,75){\makebox(0,0){${\mathcal T}\,leg_1=Leg_1$}}
\put(5,135){\makebox(0,0){${\mathcal T}^EE$}}
\put(22,135){\vector(1,0){200}}
\put(243,135){\makebox(0,0){$E=E_0$}}
\put(235,125){\vector(0,-1){110}}
\put(250,75){\makebox(0,0){$leg_1$}}

\put(70,45){\makebox(0,0){$({\mathcal T}^EM_1)_{k}$}}
\put(60,38){\vector(-2,-1){48}} \put(95,42){\vector(1,0){60}}
\put(155,48){\vector(-1,0){60}} \put(130,52){\makebox(0,0){$Y$}}
\put(172,45){\makebox(0,0){$M_{k+1}$}}
\put(175,38){\vector(2,-1){48}}

\put(64,108){\vector(-2,1){46}}
\put(74,100){\makebox(0,0){$({\mathcal T}^EE)_{k}$}}
\put(95,97){\vector(1,0){60}} \put(155,103){\vector(-1,0){60}}
\put(130,107){\makebox(0,0){$X$}}
\put(168,100){\makebox(0,0){$E_{k}$}}
\put(175,105){\vector(2,1){48}}

\put(70,90){\vector(0,-1){35}} \put(166,90){\vector(0,-1){35}}
\put(52,75){\makebox(0,0){$Leg_{k+1}$}}
\put(185,75){\makebox(0,0){$leg_{k+1}$}}
\end{picture}

\vspace{0.5cm} If $\rho^\tau_k:({\mathcal T}^EE)_k\to TE_k$ is the
anchor map of the Lie algebroid $({\mathcal T}^EE)_k\to E_k$ then
the integral curves of the vector field $\rho^\tau_k(X)$ don't
satisfy, in general, the Euler-Lagrange equations for $L$. The
reason is that the section $X$ is not, in general, a SODE along
the submanifold $E_k$ of $E$. In other words, $X$ doesn't satisfy,
in general, the equation
$$SX=\Delta_{|E_k}.$$

A solution for the above problem is given in the following
theorems.

\begin{theorem}\label{th4.4}(i) The subset $S^X$ of $E_k$
defined by
$$S^X=\{e\in E_k/(SX)(e)=\Delta(e)\}$$
is a submanifold of $E_k$.

(ii) There exists a Lie subalgebroid $A^X$ of $({\mathcal
T}^EE)_k\to E_k$ (over $S^X$) such that if
$Leg_{A^X}:A^X\to({\mathcal T}^EM_1)_{k}$ and $leg_{S^X}:S^X\to
M_{k+1}$ are the restrictions of $Leg_{k+1}:({\mathcal
T}^EE)_k\to({\mathcal T}^EM_1)_{k}$ and $leg_{k+1}:E_k\to M_{k+1}$
to $A^X$ and $S^X$, respectively, then the pair
$(Leg_{A^X},leg_{S^X})$ is a Lie algebroid isomorphism.

\end{theorem}

\begin{theorem}\label{th4.5} There is a unique section $\xi^X\in\Gamma(A^X)$
satisfying the following conditions
$$(i_{\xi^X}\omega_L)_{|S^X}=(d^{{\mathcal
T}^EE}E_L)_{|S^X},\;\;\;S(\xi^X)=\Delta_{|S^X}.$$
\end{theorem}

\begin{theorem}\label{th4.6} If $\rho_{A^X}:A^X\to TS^X$ is the anchor map of
the Lie algebroid $A^X\to S^X$ then the integral curves of the
vector field $\rho_{A^X}(\xi^X)$ on $S^X$ are solutions of the
Euler-Lagrange equations for $L$.
\end{theorem}

{\it Proof of Theorem \ref{th4.4}}. We consider the smooth map
$W_X:E_k\to E$ defined by
$$W_X(e)=pr_1(X(e)),$$
where $pr_1:{\mathcal T}^EE\to E$ is the restriction to ${\mathcal
T}^EE$ of the canonical projection $pr_1:E\times TE\to E$ on the
first factor.

Now, we will proceed in several steps.

\underline{First step}: We will prove that
\begin{equation}\label{firststep}
S^X\cap leg_{k+1}^{-1}(leg_{k+1}(e))=\{W_X(e)\},\mbox{ for all
}e\in E_k.
\end{equation}

If $e\in E_k$ then, using (\ref{Lioulo}), (\ref{omegaL}),
(\ref{EL}) and the fact that
$$(i_X\omega_L)_{|E_k}=(d^{{\mathcal T}^EE}E_L)_{|E_k},$$
we deduce that
$$(SX-\Delta)(e)\in Ker\omega_L(e).$$
On the other hand, it is clear that
$$(SX-\Delta)(e)\in ({\mathcal T}^E_eE)^V.$$
Thus, from (\ref{nucleodeleg}), we obtain that
$$(SX-\Delta)(e)\in Ker({\mathcal T}_eleg_1).$$
Therefore, if $\rho^\tau:{\mathcal T}^EE\to TE$ is the anchor map
of the Lie algebroid ${\mathcal T}^EE$, then
$$\rho^\tau((SX-\Delta)(e))\in Ker(T_eleg_1)=Ker(T_eleg_{k+1}).$$
This implies that $X^*=\rho^\tau(SX-\Delta_{|E_k})$ is a vector
field on $E_k$ which is vertical with respect to the submersion
$leg_{k+1}:E_k\to M_{k+1}$.

Suppose that the local expression of the section $X$ is
\begin{equation}\label{localX}
X=X^A\X_A+V^A\V_A.
\end{equation}
Then,
\begin{equation}\label{Xstar}
SX-\Delta=(X^A-y^A)\V_A\;\;\mbox{ and
}\;\;X^*=(X^A-y^A)\displaystyle\frac{\partial}{\partial y^A}.
\end{equation}

Moreover, since $X$ is $(Leg_{k+1},leg_{k+1})$-projectable, the
functions $X^A$ are constant on the fibres of $leg_{k+1}$.
Consequently, if $e\equiv(x^i_0,y^A_0)\in E_k$, it follows that
the integral curve of $X^*$ over the point $e$ is
$$s\mapsto\sigma(s)\equiv(x^i_0,X^A+e^{-s}(y^A_0-X^A)).$$
In particular,
$$\sigma(s)\equiv(x^i_0,X^A+e^{-s}(y^A_0-X^A))\in
leg_{k+1}^{-1}(leg_{k+1}(e)),\mbox{ for all }s\in\R,$$ which
implies that
$$\lim_{s\rightarrow\infty}\sigma(s)\equiv(x^i_0,X^A)\in
leg_{k+1}^{-1}(leg_{k+1}(e)).$$

Now, from (\ref{localX}), we have that
\begin{equation}\label{WXe}
W_X(e)=pr_1(X(e))\equiv(x^i_0,X^A).
\end{equation}
In addition, using (\ref{Xstar}) and (\ref{WXe}), one deduces that
$$X^*(W_X(e))=0.$$
On the other hand, if $e'\in S^X\cap leg_{k+1}^{-1}(leg_{k+1}(e))$
then, it is clear that
$$\tau(e)=\tau(e')$$
and, thus, $e'\equiv(x^i_0,y^A(e'))$. Furthermore, since $e'\in
S^X$, we obtain that
$$0=X^*(e')=(X^A-y^A(e'))\displaystyle\frac{\partial}{\partial
y^A}_{|e'}$$ which implies that
$$e'\equiv(x^i_0,X^A)\equiv W_X(e).$$

\underline{Second step}: We will prove that
$$S^X=\widetilde{W}_X(M_{k+1}),$$
where $\widetilde{W}_X:M_{k+1}\to E_k$ is a section of the
submersion $leg_{k+1}:E_k\to M_{k+1}$. Therefore, $S^X$ is a
submanifold of $E_k$ and $dim\,S^X=dim\,M_{k+1}$.

In fact, suppose that $e,e'\in E_k$ and
$$leg_{k+1}(e)=leg_{k+1}(e').$$
Then,
$$Leg_{k+1}(X(e))=Y(leg_{k+1}(e))=Y(leg_{k+1}(e'))=Leg_{k+1}(X(e')).$$
Consequently,
$$W_X(e)=pr_1(X(e))=pr_1(X(e'))=W_X(e').$$
So, we have that there exists a smooth map
$\widetilde{W}_X:M_{k+1}\to E_k$ such that the following diagram
is commutative

\begin{picture}(100,75)(-100,0)
\put(5,60){$E_k$} \put(97,60){\makebox(0,0){$E_k$}}
\put(20,63){\vector(1,0){65}} \put(8,55){\vector(0,-1){40}}
\put(-20,35){$leg_{k+1}$} \put(52,70){\makebox(0,0){$W_X$}}
\put(2,2){{$M_{k+1}$}} \put(18,15){\vector(2,1){72}}
\put(53,20){$\widetilde{W}_X$}
\end{picture}

Moreover, using (\ref{firststep}), we deduce that
$\widetilde{W}_X:M_{k+1}\to E_k$ is a section of the submersion
$leg_{k+1}:E_k\to M_{k+1}$ and
$$S^X=\widetilde{W}_X(M_{k+1}).$$

\underline{Third step}: We will prove the second part of the
theorem.

The section $\widetilde{W}_X:M_{k+1}\to E_k$ induces a map
$${\mathcal T}\widetilde{W}_X:({\mathcal
T}^EM_1)_{k}=(\rho^{\tau^*})^{-1}(TM_{k+1})\to
(\TEE)_k=(\rho^\tau)^{-1}(TE_k)$$ in such a way that the pair
$({\mathcal T}\widetilde{W}_X,\widetilde{W}_X)$ is a Lie algebroid
monomorphism. We will denote by $A^X$ the image of $({\mathcal
T}^EM_1)_{k}$ by the map ${\mathcal T}\widetilde{W}_X$. Then, it
is clear that $A^X$ is a Lie subalgebroid (over $S^X$) and the
pair $({\mathcal T}\widetilde{W}_X,\widetilde{W}_X)$ is an
isomorphism between the Lie algebroids $A^X\to S^X$ and
$({\mathcal T}^EM_1)_{k}\to M_{k+1}$. In fact, the inverse
morphism is the pair $(Leg_{A^X},leg_{S^X})$.

\hfill$\Box$

{\it Proof of Theorem \ref{th4.5}}. We consider the section
$\xi^X\in\Gamma(A^X)$ defined by
$$\xi^X={\mathcal T}\widetilde{W}_X\circ Y\circ leg_{S^X}.$$
Using (\ref{HamilY}) and the fact that
\begin{equation}\label{xiproject}
Leg_{A^X}\circ\xi^X=Y\circ leg_{S^X},
\end{equation}
it follows that
$$(i_{\xi^X}\omega_L)_{|S^X}=(d^{\TEE}E_L)_{|S^X}.$$
Now, from (\ref{xiproject}), we have that
$$(\xi^X-X)(e)\in Ker(Leg_{k+1})_{|({\mathcal
T}^E_eE)_k}=Ker({\mathcal T}_eleg_1),\;\mbox{ for all }e\in S^X,$$
which implies that (see (\ref{nucleodeleg}))
$$(S\xi^X)(e)=(SX)(e)=\Delta(e),\;\mbox{ for all }e\in S^X.$$

Next, suppose that $\eta$ is another section of the Lie algebroid
$A^X\to S^X$ such that
$$(i_\eta\omega_L)_{|S^X}=(d^{\TEE}E_L)_{|S^X},\;\;\;S\eta=\Delta_{|S^X}.$$
Then, it is clear that
$$(\eta-\xi^X)(e)\in Ker\omega_L(e)\cap({\mathcal
T}^E_eE)^V,\;\mbox{ for all }e\in S^X,$$ and, using
(\ref{nucleodeleg}), we deduce that
$$(\eta-\xi^X)(e)\in Ker({\mathcal
T}_eleg_1)=Ker(Leg_{k+1})_{|({\mathcal T}^E_eE)_k},\;\mbox{ for
all }e\in S^X.$$ Thus,
$$Leg_{A^X}((\eta-\xi^X)(e))=0$$
and, since $Leg_{A^X}:A^X\to({\mathcal T}^EM_1)_{k}$ is a vector
bundle isomorphism, we conclude that
$$\eta=\xi^X.$$

\hfill$\Box$

{\it Proof of Theorem \ref{th4.6}}. The section $\xi^X$ is a SODE
along the submanifold $S^X$. Therefore, from (\ref{Lioulo}) and
(\ref{endverlo}), we have that the local expression of $\xi^X$ is
\begin{equation}\label{exprexiX}
\xi^X=y^B{\X_B}_{|S^X}+\xi^B{\V_B}_{|S^X}.
\end{equation}

On the other hand, using (\ref{omegaL}), (\ref{EL}),
(\ref{exprexiX}) and the fact that
$(i_{\xi^X}\omega_L)_{|S^X}=(d^{\TEE}E_L)_{|S^X}$, it follows that
$$\xi^B\displaystyle\frac{\partial^2L}{\partial y^A\partial
y^B}+y^B\rho_B^i\frac{\partial^2L}{\partial x^i\partial
y^A}+\frac{\partial L}{\partial
y^C}\C_{AB}^Cy^B-\rho_A^i\frac{\partial L}{\partial
x^i}=0,\;\mbox{ for all }A.$$ Now, the local expression of the
vector field $\rho_{A^X}(\xi^X)$ on $S^X$ is
$$\rho_{A^X}(\xi^X)=\rho_B^iy^B\displaystyle\frac{\partial}{\partial
x^i}+\xi^B\frac{\partial}{\partial y^B}.$$

Consequently, the integral curves of $\rho_{A^X}(\xi^X)$ satisfy
the following equations
$$\displaystyle\frac{dx^i}{dt}=\rho_A^iy^A,\;\;\;\frac{d}{dt}\Big(\frac{\partial
L}{\partial y^A}\Big)+\frac{\partial L}{\partial
y^C}\C_{AB}^Cy^B-\rho_A^i\frac{\partial L}{\partial
x^i}=0,\;\mbox{ for all }i\mbox{ and }A,$$ which are the
Euler-Lagrange equations for $L$.

\hfill$\Box$

\begin{remark} If we apply the results obtained in this Section
for the particular case when the Lie algebroid $E$ is $TQ$, we
recover the results proved in \cite{Go,GoNe} for standard singular
Lagrangian systems.
\end{remark}

\begin{example}
{\rm To illustrate the theory we will consider a variation of an
example of singular lagrangian with symmetry. This example
corresponds to a mechanical model of field theories due to Capri
and Kobayashi (see \cite{CaKo1,CaKo2}).

Consider the lagrangian function
\[
L=\frac{1}{2}m_2\left( \dot{x}_2^2+\dot{y}_2^2\right)+ \dot{y}_2
x_2-\dot{x}_2y_2- x_1^2-y_1^2-x_2^2-y_2^2.
\]
The configuration space is $\widetilde{Q}=\R^4$ with local
coordinates $(x_1, y_1, x_2, y_2)$. Clearly, the lagrangian is
singular; in fact, since
\[
\omega_L=m_2dx_2\wedge d\dot{x}_2+ m_2dy_2\wedge
d\dot{y}_2+2dy_2\wedge dx_2,
\]
then
\[
\hbox{Ker }\omega_L=\hbox{span}\left\{ \frac{\partial }{\partial
x_1}, \frac{\partial }{\partial y_1}, \frac{\partial }{\partial
\dot{x}_1}, \frac{\partial}{\partial \dot{y}_1}\right\}.
\]

The system is invariant by the $S^1$ action
\[
\begin{array}{rcl}
S^1\times \widetilde{Q}& \longrightarrow & \widetilde{Q}\\
(\alpha,(x_1, y_1, x_2, y_2))&\longmapsto&(x_1, y_1,
x_2\cos\alpha-y_2\sin\alpha, x_2\sin\alpha+y_2\cos\alpha)
\end{array}
\]
Note that the action of $S^1$ on the open subset
$Q=\R^2\times(\R^2-\{(0,0)\})$ is free and then we may consider
the reduced space $Q/S^1$ and the Atiyah algebroid
$\tau_Q|S^1:TQ/S^1\to Q/S^1$. Taking polar coordinates
$x_2=\rho\cos\theta$ and $y_2=\rho\sin \theta$, we have that the
canonical projection $\pi:Q\to M=Q/S^1$ is given by
\[
\pi(x_1, y_1, \rho, \theta)=(x_1, y_1, \rho).
\]

It is clear that the Atiyah algebroid is isomorphic to the vector
bundle $TM\times \R\to M$.

On the other hand, a local basis of $S^1$-invariant vector fields
on $Q$ is $\{\frac{\partial}{\partial
x_1},\frac{\partial}{\partial y_1},\frac{\partial}{\partial\rho},$
$ \frac{\partial}{\partial\theta}\}$. These vector fields induce a
local basis of sections $\{e_1,e_2,e_3,e_0\}$ of the Atiyah
algebroid $\tau_Q|S^1:TQ/S^1\to M=Q/S^1$. Moreover, if
$(\lcf\cdot,\cdot\rcf,\rho)$ is the Lie algebroid structure on
$\tau_Q|S^1:TQ/S^1\to M=Q/S^1$, it follows that $\lcf
e_i,e_j\rcf=0$, for all $i$ and $j$, and
$$\rho(e_1)=\frac{\partial}{\partial x_1},\;\;\rho(e_2)=\frac{\partial}{\partial y_1},\;\;
\rho(e_3)=\frac{\partial}{\partial\rho},\;\;\rho(e_0)=0.$$ Now, if
$(x_1, y_1, \rho, \dot{x}_1, \dot{y}_1, \dot{\rho}, r)$ are the
local coordinates on $TQ/S^1$ induced by the local basis
$\{e_1,e_2,e_3,e_0\}$, then the reduced lagrangian is
\[
l=\frac{1}{2}m_2\left( \dot{\rho}^2+(\rho r)^2\right)+ \rho^2 r-
x_1^2-y_1^2-\rho^2.
\]
Thus, the Euler-Lagrange equations for $l$ are: \begin{eqnarray*}
m_2\ddot{\rho}-(m_2r+2)\rho r +2\rho&=&0,\\
m_2 r\rho^2+\rho^2&=&\hbox{constant},\\
x_1&=&0,\\
y_1&=&0.
\end{eqnarray*}

The local basis $\{e_1,e_2,e_3,e_0\}$ of $\Gamma(TQ/S^1)$ induces
a local basis
\[
\{\X_1, \X_2, \X_{3}, \X_0, \V_1, \V_2, \V_{3}, \V_0\}
\]
of $\Gamma(\T^{TQ/S^1}(TQ/S^1))$. The presymplectic 2-section
$\omega_l$ is written as
\[
\omega_l=m_2 \X^3\wedge \V^3+m_2\rho^2 \X^0\wedge
\V^0-\rho(m_2r+2)\X^3\wedge \X^0.
\]
The energy function is
\[
E_l=\frac{1}{2}m_2\left( \dot{\rho}^2+(\rho r)^2\right)+
x_1^2+y_1^2+\rho^2
\]
and
\[
d^{\T^{TQ/S^1}TQ/S^1}E_l= m_2\dot{\rho} \V^3+\rho^2m_2r\V^0
+(m_2\rho r^2+2\rho)\X^3+2x_1 \X^1+2y_1\X^2.
\]
Thus $\ker\omega_l=\{\X_1,\X_2, \V_1, \V_2\}$ and the primary
constraint submanifold $E_1\subset E=E_0$ is determined by the
vanishing of the constraints functions: $x_1=0,$ $y_1=0$. Now
$((\rho^{\tau_Q|S^1})^{-1}(T E_1))^\perp=\ker \omega_l$, and
therefore $E_f=E_1$.

Any solution $X\in ({\mathcal T}^E E)_1$  of the dynamical
equation \begin{equation}\label{xop}
(i_{X}\omega_l)_{|E_1}=(d^{\T^{TQ/S^1}TQ/S^1}E_l)_{|E_1}
\end{equation} is of the form: \[ X=\dot{\rho}\X_{3}+r\X_0+\tilde{f}
\V_1+\tilde{g} \V_2+
\frac{(m_2r+2)-2\rho}{m_2}\V_{3}-\frac{2\dot{\rho}(m_2r+1)}{m_2\rho}\V_0,
\] where $\tilde{f}$ and $\tilde{g}$ are arbitrary
functions on $E_1$.

The Legendre transformation $leg_l$ is in this particular case:
\[
leg_l(x_1, y_1, \rho; \dot{x}_1, \dot{y}_1, \dot{\rho}, r)=(x_1,
y_1, \rho; 0, 0, m_2\dot{\rho}, m_2\rho^2 r+\rho^2)
\]
Therefore, the submanifold $M_1=leg_l(TQ/S^1)$ of $T^*Q/S^1$ is
defined the constraints $p_{x_1}=0$ and $p_{x_2}=0$ where we
choose coordinates $(x_1, y_1, \rho; p_{{x}_1}, p_{{y}_1},
p_{\rho}, p_r)$ on $T^*Q/S^1$.

In the induced coordinates  $(x_1, y_1, \rho;  p_{\rho}, p_r)$ on
$M_1$, the hamiltonian $h: M_1\to \R$ is:
\[
h(x_1, y_1, \rho;  p_{\rho}, p_r)=\frac{1}{2m_2}
p_{\rho}^2+\frac{1}{2m_2\rho^2}
(p_{r}-\rho^2)^2+x_1^2+y_1^2+\rho^2.
\]

Applying the constraint algorithm to the presymplectic Lie
algebroid  $(\T^E M_1, \Omega_1,$ $d^{\T^E M_1} h)$ we deduce that
the final constraint submanifold $M_2$ is determined by
\[
M_2=\{(x_1, y_1, \rho;  p_{\rho}, p_r) \in M_1\; |\; x_1=0,
y_1=0\}
\]
where there exists a well defined solution of
\[
(i_{Y}\Omega_1)_{|M_{2}}=(d^{{\mathcal T}^EM_1}h)_{|M_{2}}.
\]

Now, we return to the lagrangian picture and we study the SODE
problem. Observe first that all the solutions $X$ of Equation
(\ref{xop}) are $(Leg_2, leg_2)$-projectable. Now, if  we
additionally impose the condition $SX=\Delta_{|E_1}$, that is,
\[
\dot{\rho}\V_{3}+r\V_0=\dot{x}_1
\V_1+\dot{y}_1\V_2+\dot{\rho}\V_{3}+r\V_0
\]
along $E_1$, we obtain that the submanifold  $S^X\subseteq E_1$
is uniquely defined as
\[
S^X=\{(0,0, \rho; 0, 0, \dot{\rho}, r)\in TQ/S^1\},
\]
and, therefore, the section $\xi^X$ is the SODE defined by:
\[
\xi^X=\dot{\rho}\X_{3}{_{|S^X}}+r\X_0{_{|S^X}}+
\frac{(m_2r+2)-2\rho}{m_2}\V_{3}{_{|S^X}}-\frac{2\dot{\rho}(m_2r+1)}{m_2\rho}{\V_0}_{|S^X}.
\]

}
\end{example}

\section{Vakonomic mechanics on Lie algebroids}\label{sec:vak}
\subsection{Vakonomic equations and vakonomic bracket}\label{subsec:vak-bracket}
Let $\tau :E\to Q$ be a Lie algebroid of rank $n$ over a manifold
$Q$ of dimension $m$ and $L:E\to \R$ be a Lagrangian function on
$E$. Moreover, let $M\subset E$ be an embedded submanifold of
dimension $n+m-\bar{m}$ such that $\tau _M=\tau _{|M}:M\to Q$ is a
surjective submersion.

Now, suppose that $e$ is a point of $M$, $\tau_M(e)=x\in Q$, that
$(x^i)$ are local coordinates on an open subset $U$ of $Q$, $x\in
U$, and that $\{e_A\}$ is a local basis of $\Gamma(E)$ on $U$.
Denote by $(x^i,y^A)$ the corresponding local coordinates for $E$
on the open subset $\tau^{-1}(U)$. Assume that
$$M\cap\tau^{-1}(U)\equiv\{(x^i,y^A)\in\tau^{-1}(U) \,|\,
\Phi^\alpha(x^i,y^A)=0,\;\alpha=1,\dots,\bar{m}\}$$ where
$\Phi^\alpha$ are the local independent constraint functions for
the submanifold $M$. The rank of the $(\bar{m}\times(n+m))$-matrix
$$\Big(\displaystyle\frac{\partial\Phi^\alpha}{\partial
x^i},\frac{\partial\Phi^\alpha}{\partial y^A}\Big)$$ is maximun,
that is, $\bar{m}$. On the other hand, since $\tau_M:M\to Q$ is a
submersion, we deduce that there exists $v_i\in T_eM$ such that
$(T\tau_M)(v_i)={\frac{\partial}{\partial x^i}}_{|x}$, for all
$i\in\{1,\dots,m\}$. Thus,
$$v_i={\displaystyle\frac{\partial}{\partial
x^i}}_{|e}+v_i^A{\frac{\partial}{\partial y^A}}_{|e}$$ which
implies that
$$\displaystyle{\frac{\partial\Phi^\alpha}{\partial
x^i}}_{|e}=-v^A_i{\frac{\partial\Phi^\alpha}{\partial
y^A}}_{|e},\mbox{ for }\alpha\in\{1,\dots,\bar{m}\}\mbox{ and
}i\in\{1,\dots,m\}.$$

Therefore, the rank of the matrix
$$\Big(\displaystyle{\frac{\partial\Phi^\alpha}{\partial
y^A}}_{|e}\Big)_{\alpha=1,\dots,\bar{m};A=1,\dots,n}$$

\noindent is $\bar{m}$. We will suppose, without the loss of
generality, that the $(\bar{m}\times\bar{m})$-matrix
$$\Big(\displaystyle{\frac{\partial\Phi^\alpha}{\partial
y^B}}_{|e}\Big)_{\alpha=1,\dots,\bar{m};B=1,\dots,\bar{m}}$$ is
regular. Then, we will use the following notation
$$y^A=(y^\alpha,y^a),$$
for $1\leq A\leq n$, $1\leq\alpha\leq\bar{m}$ and $\bar{m}+1\leq
a\leq n$.

Now, using the implicit function theorem, we obtain that there
exist an open subset $\widetilde{V}$ of $\tau^{-1}(U)$, an open
subset $W\subseteq\R^{m+n-\bar{m}}$ and smooth real functions
$$\Psi^\alpha:W\to\R,\;\;\alpha=1,\dots,\bar{m},$$
such that
$$M\cap\widetilde{V} \equiv \{(x^i,y^A)\in\widetilde{V} \,|\,
y^\alpha=\Psi^\alpha(x^i,y^a),\;\alpha=1,\dots,\bar{m}\}.$$
Consequently, $(x^i,y^a)$ are local coordinates on $M$. We will
denote by $\tilde{L}$ the restriction of $L$ to $M$.

Now, we will develop a geometric description of vakonomic
mechanics on Lie algebroids, naturally generalizing the previous
results of the third author and co\-lla\-borators \cite{CLMM}.
Moreover, for the case $M=E$ we also generalize the formulation
given by Skinner and Rusk  \cite{SR1,SR2} for singular lagrangians
to general Lie algebroids.

Consider the Whitney sum of $E^*$ and $E$, $E^*\oplus E$, and the
canonical projections $pr_1: E^*\oplus E \longrightarrow E^*$ and
$pr_2: E^*\oplus E \longrightarrow E$. Now, let $W_0$ be the
submanifold  $W_0=pr_2^{-1} (M)=E^*\times _Q M$  and the
restrictions $\pi_1={pr_1}_{|W_0}$ and $\pi_2={pr_2}_{|W_0}$. Also
denote by $\nu: W_0\longrightarrow Q$ the canonical projection.
The following diagrams illustrate the situation

\begin{picture}(375,65)(20,40)
\put(97,85){$E^*\oplus E$} \put(75,40){\makebox(0,0){$E^*$}}
\put(110,80){\vector(-1,-1){30}} \put(120,80){\vector(1,-1){30}}
\put(155,40){\makebox(0,0){$E$}} \put(75,65){\makebox(0,0){pr$_1$}}
\put(155,65){\makebox(0,0){pr$_2$}}
\put(260,80){\vector(-1,-1){30}} \put(270,80){\vector(1,-1){30}}
\put(247,85){$E^*\times _Q M$} \put(225,40){\makebox(0,0){$E^*$}}
\put(305,40){\makebox(0,0){$M$}} \put(225,65){\makebox(0,0){$\pi
_1$}} \put(305,65){\makebox(0,0){$\pi _2$}}
\end{picture}

\vspace{.2cm} Next, we consider the prolongation of the Lie
algebroid $E$ over $\tau ^*:E^*\to Q$ (res\-pectively, $\nu
:W_0\to Q$). We will denote this Lie algebroid by ${\mathcal
T}^EE^*$ (respectively, ${\mathcal T}^E W_0$). Moreover, we can
prolong $\pi _1:W_0\to E^*$ to a morphism of Lie algebroids
${\mathcal T}\pi _1 :{\mathcal T}^E W_0\to {\mathcal T}^EE^*$
defined by $\T\pi_1=(Id,T\pi_1)$.

If $(x^i,p_A)$ are the local coordinates on $E^*$ associated with
the local basis $\{e_A\}$ of $\Gamma(E)$, then $(x^i,p_A,y^a)$ are
local coordinates for $W_0$ and we may consider the local basis
$\{{\mathcal Y}_A ,{\mathcal P}^A ,\mathcal{V}_a\}$ of
$\Gamma({\mathcal T}^{E}W_{0})$ defined by
$$\begin{array}{rcl} {\mathcal Y}_A(e^*, \check{e})&=&(e_A(x),
\rho_A^i\displaystyle\frac{\partial
}{\partial x^i}_{|e^*}, 0),\\
{\mathcal P}^A(e^*,\check{e})&=&(0,\displaystyle\frac{\partial
}{\partial
p_{A}}_{|e^*}, 0),\\
{\mathcal V}_a(e^*, \check{e})&=&(0, 0,\displaystyle\frac{\partial
}{\partial y^{a}}_{|\check{e}}),
\end{array}$$
where $(e^*, \check{e})\in W_0$ and $\nu(e^*,\check{e})=x$. If
$(\lcf \cdot , \cdot \rcf ^{\nu}, \rho^{\nu})$ is the Lie
algebroid structure on ${\mathcal T}^{E}W_{0}$, we have that
\[
\lcf {\mathcal Y}_{A}, {\mathcal Y}_{B} \rcf^{\nu} = {\mathcal
C}_{AB}^{C} {\mathcal Y}_{C},
\]
and the rest of the fundamental Lie brackets are zero. Moreover,
\[
\rho^{\nu}({\mathcal Y}_{A}) = \rho_{A}^{i} \displaystyle
\frac{\partial}{\partial x^{i}}, \; \; \rho^{\nu}({\mathcal
P}^{A}) = \displaystyle \frac{\partial}{\partial p_{A}}, \; \;
\rho^{\nu}({\mathcal V}_{a}) = \displaystyle
\frac{\partial}{\partial y^{a}}.
\]

The Pontryagin Hamiltonian $H_{W_0}$ is a function in
$W_0=E^*\times _Q M$ given by
\[
H_{W_0} (e^* , \check{e})= \langle e^* , \check{e}\rangle
-\tilde{L}(\check{e}),
\]
or, in local coordinates,
\begin{equation}\label{H0}
H_{W_0}(x^i, p_A, y^a)=p_ay^a+p_{\alpha}\Psi^{\alpha}(x^i,
y^a)-\tilde{L}(x^i, y^a)\, .
\end{equation}

Moreover, one can consider the presymplectic 2-section $\Omega
_0=({\mathcal T}\pi _1, \pi_1)^*\Omega _E$, where $\Omega _E$ is
the canonical symplectic section on ${\mathcal T}^EE^*$ defined in
Equation (\ref{sym}). In local coordinates,
\begin{equation}\label{Omega0}
\Omega_0=\Y^A\wedge {\mathcal P}_A + \frac{1}{2} {\mathcal
C}_{AB}^C p_C \Y^A\wedge \Y^B.
\end{equation}

Therefore, we have the triple $({\mathcal T}^{E}W_0,\Omega
_0,d^{{\mathcal T}^{E}W_0}H_{W_0})$ as a presymplectic hamiltonian
system.

\begin{definition}
The vakonomic problem on Lie algebroids is find the solutions for
the equation
\begin{equation}\label{Hamilt}
i_{X}\Omega _0 = d^{{\mathcal T}^{E}W_0} H_{W_0},
\end{equation}
that is, to solve the constraint algorithm for $({\mathcal
T}^{E}W_0,\Omega _0,d^{{\mathcal T}^{E}W_0} H_{W_0})$.
\end{definition}

In local coordinates, we have that
\[
d^{{\mathcal T}^{E}W_0}H_{W_0} = (p_\alpha\frac{\partial \Psi
^\alpha}{\partial x^i} -\frac{\partial \tilde{L}}{\partial
x^i})\rho ^i_A \Y^A + \Psi ^\alpha \P_\alpha +y^a \P_a +
(p_a+p_\alpha \frac{\partial \Psi ^\alpha}{\partial
y^a}-\frac{\partial \tilde{L}}{\partial y^a})\V^a.
\]

If we apply the constraint algorithm,
\[
W_1=\{ w\in E^*\times _Q M\, | \, d^{{\mathcal T}^{E}W_0}
H_{W_0}(w)(Y)=0,\quad \forall Y\in Ker\,\Omega _0(w)\}.
\]
Since $Ker\,\Omega _0= span \{ {\mathcal V}_a \}$, we get that $W_1$
is locally characterized by the equations
\[
\varphi _a = d^{{\mathcal T}^{E}W_0} H_{W_0} ({\mathcal
V}_a)=p_a+p_\alpha \frac{\partial \Psi ^\alpha}{\partial
y^a}-\frac{\partial \tilde{L}}{\partial y^a}=0,
\]
or
\[
p_a=\frac{\partial \tilde{L}}{\partial y^a}-p_\alpha
\frac{\partial \Psi ^\alpha}{\partial y^a}, \quad \bar{m}+1\leq a
\leq n .
\]
Let us also look for the expression of $X$ satisfying Eq.
(\ref{Hamilt}). A direct computation shows that
\[
X= y^a{\mathcal Y}_a + \Psi ^\alpha {\mathcal Y}_\alpha + \Big [
\Big ( \frac{\partial \tilde{L}}{\partial x^i}-p_\alpha
\frac{\partial \Psi ^\alpha}{\partial x^i} \Big ) \rho ^i_A
-y^a{\mathcal C}^B_{Aa}p_B-\Psi ^\alpha {\mathcal
C}^B_{A\alpha}p_B \Big ]{\mathcal P}^A + \Upsilon ^a{\mathcal
V}_a.
\]
Therefore, the vakonomic equations are
$$\left \{
\begin{array}{l}
\displaystyle \dot{x}^i=y^a \rho ^i_a+\Psi ^\alpha \rho ^i_\alpha ,\\[10pt]
\displaystyle \dot{p}_{\alpha}=\Big ( \frac{\partial
\tilde{L}}{\partial x^i}-p_\beta \frac{\partial \Psi
^\beta}{\partial x^i} \Big ) \rho ^i_{\alpha} -y^a {\mathcal
C}^B_{\alpha a}p_B-\Psi ^\beta {\mathcal
C}^B_{\alpha\beta}p_B ,\\[10pt]
\displaystyle \frac{d}{dt}\left( \frac{\partial \tilde{L}}{\partial
y^a}-p_\alpha \frac{\partial \Psi ^\alpha}{\partial y^a} \right)=
\Big ( \frac{\partial \tilde{L}}{\partial x^i}-p_\alpha
\frac{\partial \Psi ^\alpha}{\partial x^i} \Big ) \rho ^i_{a} -y^b
{\mathcal
C}^B_{a b}p_B-\Psi ^\alpha {\mathcal C}^B_{a\alpha}p_B.\\[10pt]
\end{array}
\right .$$
\begin{remark}
We note that the vakonomic equations can be obtained following a
constrained variational principle. Below, we will show this variational
way of obtaining these equations (see Subsection \ref{subsec:variational}).
\end{remark}

Of course, we know that there exist sections $X$ of ${\mathcal
T}^{E}W_0$ along $W_1$ satisfying (\ref{Hamilt}), but they may not
be sections of $(\rho^\nu)^{-1}(TW_1)={\mathcal T}^{E}W_1$, in
general. Then, following the procedure detailed in Section
\ref{pre}, we obtain a sequence of embedded submanifolds
\[
\ldots \hookrightarrow W_{k+1}\hookrightarrow W_{k}\hookrightarrow
\ldots \hookrightarrow W_{2}\hookrightarrow W_{1}\hookrightarrow
W_{0}=E^*\times _Q M.
\]
If the algorithm stabilizes, then we find a final constraint
submanifold $W_f$ on which at least a section $X\in\Gamma(
{\mathcal T}^{E}W_f)$ verifies
$$(i_{X}\Omega _0 = d^{{\mathcal T}^{E}W_0} H_{W_0})_{|W_f}.$$

As in \cite{CLMM}, we analyze the case when $W_f=W_1$. Consider
the restriction $\Omega_{1}$ of $\Omega_{0}$ to ${\mathcal
T}^{E}W_1$.
\begin{proposition}\label{prop:charact-sympl}
$\Omega_{1}$ is a symplectic section of the Lie algebroid
${\mathcal T}^{E}W_1$ if and only if for any system of coordinates
$(x^i, p_A, y^a)$ on $W_0$ we have that
\[
\det\left( \frac{\partial^2 \tilde{L}}{\partial y^a \partial
y^b}-p_{\alpha}\frac{\partial^2 {\Psi^{\alpha}}}{\partial y^a
\partial y^b}\right)\not=0,\mbox{ for all point in }W_1.
\]
\end{proposition}
\begin{proof}
It is clear that $d^{{\mathcal T}^{E}W_1}\Omega_1=0$.

On the other hand, if $w\in W_1$ then, since the elements
$(d^{{\mathcal T}^EW_0}\varphi_a)(w)$ are independent in
$({\mathcal T}^E_wW_0)^*$, we have that
$$({\mathcal T}^E_wW_1)^\circ=<(d^{{\mathcal
T}^EW_0}\varphi_a)(w)>$$ (note that $dim\;({\mathcal
T}^E_wW_1)^\circ=n-\bar{m}$). Moreover, using a well-known result
(see, for instance, \cite{LeRo}), we deduce that
\begin{equation}\label{for}
dim\;({\mathcal T}^E_wW_1)^{\perp,\Omega_0}=dim\;({\mathcal
T}^E_wW_0)-dim\;({\mathcal
T}^E_wW_1)+dim\;(Ker\Omega_0(w)\cap{\mathcal T}^E_wW_1).
\end{equation}
Now, suppose that $\Omega_1$ is a nondegenerate 2-section on
${\mathcal T}^E_wW_1\to W_1$. Then, it is clear that
$$Ker\Omega_0(w)\cap{\mathcal T}^E_wW_1=\{0\}.$$
Thus, the matrix $(d^{{\mathcal T}^{E}W_0}\varphi_a)(w) (\V_b(w))$
is regular, that is,
\[
\det\left( {\frac{\partial^2 \tilde{L}}{\partial y^a \partial
y^b}}_{|w}-p_{\alpha}(w){\frac{\partial^2
{\Psi^{\alpha}}}{\partial y^a
\partial y^b}}_{|w}\right)\not=0.
\]

Conversely, assume that the matrix $(d^{{\mathcal
T}^{E}W_0}\varphi_a)(w) (\V_b(w))$ is regular, then
$$Ker\Omega_0(w)\cap{\mathcal T}^E_wW_1=\{0\}.$$
This implies that (see (\ref{for}))
$$dim\;({\mathcal T}^E_wW_1)^{\perp,\Omega_0}=dim\;({\mathcal
T}^E_wW_0)-dim\;({\mathcal T}^E_wW_1)=dim\;(Ker\Omega_0(w)).$$
Therefore, since $Ker\Omega_0(w)\subseteq({\mathcal
T}^E_wW_1)^{\perp,\Omega_0}$, it follows that
$$Ker\Omega_0(w)=({\mathcal
T}^E_wW_1)^{\perp,\Omega_0},$$ and, consequently,
$$({\mathcal T}^E_wW_1)\cap({\mathcal
T}^E_wW_1)^{\perp,\Omega_0}=\{0\},$$ that is, $\Omega_1(w)$ is a
nondegenerate 2-section.
\end{proof}

In what follows, we will use the following notation
\[
{\mathcal R}_{ab} = \displaystyle \frac{\partial
\tilde{L}}{\partial y^{a} \partial y^{b}} - p_{\alpha}
\frac{\partial^{2} \Psi^{\alpha}}{\partial y^{a} \partial y^{b}},
\; \; \mbox{ for all } a \mbox{ and } b.
\]

\begin{remark} Suppose that the submanifold $M$ is a real
vector subbundle $D$ (over Q) of $E$ (note that $\tau_{D} =
\tau_{|D}: D \to M$ is a surjective submersion). Then, we may
consider local coordinates $(x^{i})$ on an open subset $U$ of $Q$
and a local basis $\{e_{\alpha}, e_{a} \}$ of $\Gamma(E)$ on $U$
such that $\{e_{a}\}$ is a local basis of $\Gamma(D)$. Thus, if
$(x^{i}, y^{\alpha}, y^{a})$ are the corresponding local
coordinates on $\tau^{-1}(U)$, we have that
\[
\tau_{D}^{-1}(U) = D \cap \tau^{-1}(U) \equiv \{(x^{i},
y^{\alpha}, y^{a}) \in \tau^{-1}(U) / y^{\alpha} = 0 \}.
\]
In other words, the local function $\Psi^{\alpha}$ is the zero
function, for all $\alpha$. Therefore, in this case,
\[
{\mathcal R}_{ab} = \displaystyle \frac{\partial^{2}
\tilde{L}}{\partial y^{a} \partial y^{b}}, \; \; \mbox{ for all }
a \mbox{ and } b.
\]
We remark that the condition
\[
\det\left( \displaystyle \frac{\partial^2 \tilde{L}}{\partial y^a
\partial y^b}\right)\not=0
\]
implies that the corresponding nonholonomic problem determined by
the pair $(L, D)$ has a unique solution (see \cite{CoLeMaMa}).
\end{remark}

\begin{remark}\label{subW1}
We remark that the condition $\det\left( {\mathcal R}_{ab}
\right)\not=0$ implies that the matrix $\left( \displaystyle
\frac{\partial \varphi_{a}}{\partial y^{b}} \right)_{a, b =
\bar{m}+1, \dots , n}$ is regular. Thus, using the implicit
theorem function, we deduce that $(x^{i}, p_{A}, y^{a})$ are local
coordinates for $W_{0}$ on an open subset $A_{0} \subseteq W_{0}$
in such a way that there exist an open subset $\tilde{W} \subseteq
\R^{m+n}$ and smooth real functions
\[
\mu^{a}: \tilde{W} \to \R, \; \; \; a = \bar{m} +1, \dots, n,
\]
such that
\begin{equation}\label{ImplW1}
W_{1} \cap A_0 = \{(x^{i}, p_{A}, y^{a}) \in A_0 / y^{a} =
\mu^{a}(x^i, p_{A}), \; a = \bar{m} + 1, \dots, n \}.
\end{equation}
Therefore, a local basis of $\Gamma ({\mathcal T}^{E}W_{1})$ is
given by
\[
\{ {\mathcal Y}_{A1} = ({\mathcal Y}_{A} + \rho^{i}_{A}
\displaystyle \frac{\partial \mu^{a}}{\partial x^{i}} {\mathcal
V}_{a})_{|W_{1}}, {\mathcal P}_{1}^{A} = ({\mathcal P}^{A} +
\displaystyle \frac{\partial \mu^{a}}{\partial p_{A}} {\mathcal
V}_{a})_{|W_{1}} \}.
\]
This implies that
\[
\{ {\mathcal Y}_{A1}, {\mathcal P}_{1}^{A}, ({\mathcal
V}_{a})_{|W_{1}} \}
\]
is a local basis of $\Gamma({\mathcal T}^{E}_{W_{1}}W_{0})$. Note
that, from (\ref{ImplW1}), it follows that $(x^{i}, p_{A})$ are
local coordinates on $W_{1}$. Moreover, if $\nu_{1}: W_{1} \to Q$
is the canonical projection and $(\lcf \cdot , \cdot
\rcf^{\nu_{1}}, \rho^{\nu_{1}})$ is the Lie algebroid structure on
${\mathcal T}^{E}W_{1} \to W_{1}$, we have that
\[
\lcf {\mathcal Y}_{A1}, {\mathcal Y}_{B1} \rcf^{\nu_{1}} =
C_{AB}^{C} {\mathcal Y}_{C1}
\]
and the rest of the fundamental Lie brackets are zero. In
addition,
\begin{equation}\label{rhonu1}
\rho^{\nu_{1}}({\mathcal Y}_{A1}) = \rho^{i}_{A} \displaystyle
\frac{\partial}{\partial x^{i}}, \; \; \rho^{\nu_{1}}({\mathcal
P}_{1}^{A}) = \displaystyle \frac{\partial}{\partial p_{A}}.
\end{equation}
\end{remark}

Now, we will prove the following result.

\begin{theorem}\label{unic}
If $\Omega_{1}$ is a symplectic $2$-section on the Lie algebroid
${\mathcal T}^{E}W_{1} \to W_{1}$ then there exists a unique
section $\xi_{1}$ of ${\mathcal T}^{E}W_{1} \to W_{1}$ whose
integral curves are solutions of the vakonomic equations for the
system $(L, M)$. In fact, if $H_{W_{1}}$ is the restriction to
$W_{1}$ of the Pontryagin Hamiltonian $H_{W_{0}}$, then $\xi_{1}$
is the Hamiltonian section of $H_{W_{1}}$ with respect to the
symplectic section $\Omega_{1}$, that is,
$$i_{\xi_{1}}\Omega_{1} = d^{{\mathcal T}^{E}W_{1}}H_{W_{1}}.$$
\end{theorem}

\begin{proof} We have that (see (\ref{H0}))
\[
(d^{{\mathcal T}^{E}W_{0}}H_{W_{0}})({\mathcal V}_{a}) =
\rho^{\nu}({\mathcal V}_{a})(H_{W_{0}}) = \displaystyle \left(
\frac{\partial H_{W_{0}}}{\partial y^{a}}\right) = \varphi_{a}.
\]
Therefore, from (\ref{Omega0}), it follows that
\[
(\Omega_{0})_{|W_{1}}(\xi_{1}, ({\mathcal V}_{a})_{|W_{1}}) =
(d^{{\mathcal T}^{E}W_{0}}H_{W_{0}})({\mathcal V}_{a})_{|W_{1}} =
0.
\]
This implies that
\[
i_{\xi_{1}}(\Omega_{0})_{|W_{1}} = (d^{{\mathcal
T}^{E}W_{0}}H_{W_{0}})_{|W_{1}}
\]
and, consequently, the integral curves of $\xi_{1}$ are solutions
of the vakonomic equations for the system $(L, M)$.

Moreover, if $\xi_{1}' \in \Gamma({\mathcal T}^{E}W_{1})$ is
another solution of the equation
\[
i_{\xi_{1}'}(\Omega_{0})_{|W_{1}} = (d^{{\mathcal
T}^{E}W_{0}}H_{W_{0}})_{|W_{1}}
\]
then it is also a solution of the equation
\[
i_{\xi_{1}'}\Omega_{1} = d^{{\mathcal T}^{E}W_{1}}H_{W_{1}}
\]
which implies that $\xi_{1}' = \xi_{1}$.
\end{proof}

Suppose that $(x^{i}, p_{A})$ are local coordinates on $W_{1}$ as
in Remark \ref{subW1} and that $\{ {\mathcal Y}_{A1}, {\mathcal
P}^{A}_{1} \}$ is the corresponding local basis of $\Gamma
({\mathcal T}^{E}W_{1})$. Then, if $\{ {\mathcal Y}^{A}_{1},
{\mathcal P}_{A1}\}$ is the dual basis of $\{ {\mathcal Y}_{A1},
{\mathcal P}^{A}_{1} \}$, we have that (see (\ref{Omega0}))
\begin{equation}\label{Omega1}
\Omega_{1} = {\mathcal Y}^{A}_{1} \wedge {\mathcal P}_{A1} +
\displaystyle \frac{1}{2} {\mathcal C}_{AB}^{C}p_{C} {\mathcal
Y}^{A}_{1} \wedge {\mathcal Y}^{B}_{1}.
\end{equation}
On the other hand, from (\ref{rhonu1}), it follows that
\[
d^{{\mathcal T}^{E}W_{1}} H_{W_{1}} = \rho^{i}_{A} \displaystyle
\frac{\partial H_{W_{1}}}{\partial x^{i}} {\mathcal Y}^{A}_{1} +
\frac{\partial H_{W_{1}}}{\partial p_{A}} {\mathcal P}_{A1}.
\]
Therefore,
\[
\xi_{1} = \displaystyle \frac{\partial H_{W_{1}}}{\partial p_{A}}
{\mathcal Y}_{A1} - ({\mathcal C}^{C}_{AB} p_{C} \displaystyle
\frac{\partial H_{W_{1}}}{\partial p_{B}} + \rho^{i}_{A}
\displaystyle \frac{\partial H_{W_{1}}}{\partial x^{i}}) {\mathcal
P}^{A}_{1}.
\]
Note that
\[
\displaystyle \frac{\partial H_{W_{1}}}{\partial x^{i}} = \left(
(\frac{\partial}{\partial x^{i}} + \frac{\partial
\mu^{a}}{\partial x^{i}} \frac{\partial}{\partial
y^{a}})(H_{W_{0}})\right)_{|W_{1}}, \;  \displaystyle
\frac{\partial H_{W_{1}}}{\partial p_{A}} = \left(
(\frac{\partial}{\partial p_{A}} + \frac{\partial
\mu^{a}}{\partial p_{A}} \frac{\partial}{\partial
y^{a}})(H_{W_{0}})\right)_{|W_{1}}
\]
which implies that
\[
\begin{array}{rcl}
\xi_{1}(x^{j}, p_{B}) &=& \mu^{a}(x^{j}, p_{B}) {\mathcal Y}_{a1}
+ \Psi^{\alpha}(x^{j}, \mu^{a}(x^{j}, p_{B})) {\mathcal Y}_{\alpha
1}\\[8pt]&& - \left({\mathcal C}_{Aa}^{b} p_{b}\mu^{a}(x^{j}, p_{B}) + \mathcal{C}^{b}_{A\alpha}
p_{b} \Psi^{\alpha}(x^{j}, \mu^{a}(x^{j}, p_{B}))\right.\\[8pt]&& +
\rho^{i}_{A} ( p_{\alpha} \displaystyle \frac{\partial
\Psi^{\alpha}}{\partial x^{i}}_{|(x^{j}, \mu^{a}(x^{j}, p_{B}))} -
\frac{\partial \tilde{L}}{\partial x^{i}}_{|(x^{j}, \mu^{a}(x^{j},
p_{B}))})\left. \right){\mathcal P}_{1}^{A}.
\end{array}
\]

Now, we will introduce the following definition.

\begin{definition}
The vakonomic system $(L, M)$ on the Lie algebroid $\tau: E \to Q$
is said to be \emph{regular} if $\Omega_{1}$ is a symplectic
$2$-section of the Lie algebroid ${\mathcal T}^{E}W_{1} \to
W_{1}$.
\end{definition}

Suppose that $(L, M)$ is a regular vakonomic system and that
$F_{1} \in C^{\infty}(W_{1})$. Then, the \emph{Hamiltonian section
of $F_{1}$} with respect to $\Omega_{1}$ is the section ${\mathcal
H}_{F_{1}}^{\Omega_{1}}$ of the Lie algebroid ${\mathcal
T}^{E}W_{1} \to W_{1}$ which is characterized by the equation
\[
i({\mathcal H}_{F{1}}^{\Omega_{1}})\Omega_{1} = d^{{\mathcal
T}^{E}W_{1}}F_{1}.
\]
Note that $\xi_{1} = {\mathcal H}^{\Omega_{1}}_{H_{W_{1}}}$.

\noindent Using the Hamiltonian sections, one may introduce a bracket of
functions $\{\cdot , \cdot \}_{(L, M)}\kern-2.1pt: C^{\infty}(W_{1}) \times
C^{\infty}(W_{1}) \to C^{\infty}(W_{1})$ as follows
\[
\{F_{1}, G_{1}\}_{(L, M)} = \Omega_{1}({\mathcal
H}_{F_{1}}^{\Omega_{1}}, {\mathcal H}^{\Omega_{1}}_{G_{1}}) =
\rho^{\nu_{1}}({\mathcal H}_{G_{1}}^{\Omega_{1}})(F_{1}),
\]
for $F_{1}, G_{1} \in C^{\infty}(W_{1})$. The bracket $\{\cdot ,
\cdot \}_{(L, M)}$ is called the \emph{ vakonomic bracket
associated with the system} $(L, M)$.

\begin{theorem}\label{vakobrac}
The vakonomic bracket $\{\cdot , \cdot \}_{(L, M)}$ associated
with a regular vakonomic system is a Poisson bracket on $W_{1}$.
Moreover, if $F_{1} \in C^{\infty}(W_{1})$ then the temporal
evolution of $F_{1}$, $\dot{F}_{1}$, is given by
\[
\dot{F}_{1} = \{F_{1}, H_{W_{1}}\}_{(L, M)}.
\]
\end{theorem}

\begin{proof}
Let $\flat_{\Omega_{1}}: {\mathcal T}^{E}W_{1} \to ({\mathcal
T}^{E}W_{1})^{*}$ be the musical isomorphism induced by
$\Omega_{1}$ which is defined by
\[
\flat_{\Omega_{1}}(X_{1}) = i(X_{1})\Omega_{1}(w_{1}), \; \;
\mbox{ for } X_{1} \in {\mathcal T}^{E}_{w_{1}}W_{1} \; \mbox{ and
} w_{1} \in W_{1}.
\]
Then, we may introduce the section $\Lambda_{1}$ of the vector
bundle $\Lambda^{2}{\mathcal T}^{E}W_{1} \to W_{1}$ as follows
\[
\Lambda_{1}(\alpha_{1}, \beta_{1}) =
\Omega_{1}(\flat^{-1}_{\Omega_{1}} \circ \alpha_{1},
\flat^{-1}_{\Omega_{1}} \circ \beta_{1}), \; \; \mbox{ for }
\alpha_{1}, \beta_{1} \in \Gamma({\mathcal T}^{E}W_{1})^{*}.
\]
If $\lcf \cdot , \cdot \rcf^{\nu_{1}}$ is the Schouten-Nijenhuis
bracket associated with the Lie algebroid ${\mathcal T}^{E}W_{1}$
$\to W_{1}$ then, using that $d^{{\mathcal T}^{E}W_{1}}\Omega_{1} =
0$, one may prove that
\[
\lcf \Lambda_{1}, \Lambda_{1} \rcf^{\nu_{1}} = 0.
\]
Thus, $\Lambda_{1}$ is a triangular matrix, the pair $({\mathcal
T}^{E}W_{1}, ({\mathcal T}^{E}W_{1})^{*})$ is a triangular Lie
bialgebroid in the sense of Mackenzie and Xu (see Section 4 in
\cite{MaXu}) and $\{\cdot , \cdot \}_{(L, M)}$ is a Poisson
bracket on $W_{1}$ (see Proposition 3.6 in \cite{MaXu}).

On the other hand, if $F_{1} \in C^{\infty}(W_{1})$ and
$\rho^{\nu_{1}}$ is the anchor map of the Lie algebroid ${\mathcal
T}^{E}W_{1} \to W_{1}$ then
\[
\dot{F}_{1} = \rho^{\nu_{1}}(\xi_{1})(F_{1}) =
\rho^{\nu_{1}}({\mathcal H}_{H_{W_{1}}}^{\Omega_{1}})(F_{1}) =
(d^{{\mathcal T}^{E}W_{1}}F_{1})({\mathcal
H}^{\Omega_{1}}_{H_{W_{1}}}).
\]
Therefore,
\[
\dot{F}_{1} = \{F_{1}, H_{W_{1}}\}_{(L, M)}.
\]
\end{proof}
Suppose that $(x^{i}, p_{A})$ are local coordinates on $W_{1}$ as
in Remark \ref{subW1} and that $\{{\mathcal Y}_{A1}, {\mathcal
P}^{A}_{1}\}$ is the corresponding local basis of
$\Gamma({\mathcal T}^{E}W_{1})$. If $F_{1} \in C^{\infty}(W_1)$
then, from (\ref{rhonu1}), it follows that
\[
d^{{\mathcal T}^{E}W_{1}} F_{1} = \rho^{i}_{A} \displaystyle
\frac{\partial F_{1}}{\partial x^{i}} {\mathcal Y}^{A}_{1} +
\frac{\partial F_{1}}{\partial p_{A}} {\mathcal P}_{A1}.
\]
Consequently, using (\ref{Omega1}), we deduce that
\[
{\mathcal H}^{\Omega_{1}}_{F_{1}} = \displaystyle \frac{\partial
F_{1}}{\partial p_{A}} {\mathcal Y}_{A1} - ({\mathcal C}^{C}_{AB}
p_{C} \displaystyle \frac{\partial F_{1}}{\partial p_{B}} +
\rho^{i}_{A} \displaystyle \frac{\partial F_{1}}{\partial x^{i}})
{\mathcal P}^{A}_{1}.
\]
This implies that
\begin{equation}\label{Poisson1}
\{F_{1}, G_{1}\}_{(L, M)} = \rho^{i}_{A} \left( \displaystyle
\frac{\partial F_{1}}{\partial x^{i}} \frac{\partial
G_{1}}{\partial p_{A}} - \frac{\partial F_{1}}{\partial p_{A}}
\frac{\partial G_{1}}{\partial x^{i}}\right) - {\mathcal
C}_{AB}^{C}p_{C} \displaystyle \frac{\partial F_{1}}{\partial
p_{A}} \frac{\partial G_{1}}{\partial p_{B}}.
\end{equation}
Now, we consider the linear Poisson structure
\[
\{\cdot , \cdot \}_{E^{*}}: C^{\infty}(E^{*}) \times
C^{\infty}(E^{*}) \to C^{\infty}(E^{*})
\]
on $E^{*}$ induced by the Lie algebroid structure on $E$ and the
canonical symplectic structure $\Omega_{E}$ on ${\mathcal
T}^{E}E^{*}$ (see Section \ref{algebroides}).

\begin{corollary}
If $(L, M)$ is a regular vakonomic system on a Lie algebroid $E$
then the restriction $(\pi_{1})_{|W_{1}}: W_{1} \to E^{*}$ of
$\pi_{1}: W_{0} \to E^{*}$ to $W_{1}$ is a local Poisson
isomorphism. Moreover, if ${\mathcal T}(\pi_{1})_{|W_{1}}:
{\mathcal T}^{E}W_{1} \to {\mathcal T}^{E}E^{*}$ is the
corresponding prolongation then the pair $({\mathcal
T}(\pi_{1})_{|W_{1}}, (\pi_{1})_{|W_{1}})$ is a local
symplectomorphism between the symplectic Lie algebroids
$({\mathcal T}^{E}W_{1}, \Omega_{1})$ and $({\mathcal T}^{E}E^{*},
\Omega_{E})$.
\end{corollary}
\begin{proof}
 From Remark \ref{subW1}, we obtain that $(\pi_{1})_{|W_{1}}: W_{1}
\to E^{*}$ is a local diffeomorphism. Furthermore, using
(\ref{Poisson}) and (\ref{Poisson1}), we deduce that
\[
\{F \circ (\pi_{1})_{|W_{1}}, G \circ (\pi_{1})_{|W_{1}} \}_{(L,
M)} = \{F, G\}_{E^{*}} \circ (\pi_{1})_{|W_{1}},
\]
for $F, G \in C^{\infty}(E^{*})$. This proves the first part of
the corollary. The second part follows from the first one and
using the fact that $({\mathcal T}(\pi_{1})_{|W_{1}},
(\pi_{1})_{|W_{1}})^{*}\Omega_{E} = \Omega_{1}$.
\end{proof}
\begin{remark}
If $(\pi_{1})_{|W_{1}}: W_{1} \to E^{*}$ is a global
diffeomorphism then we may consider the real function $H$ on
$E^{*}$ defined by
\[
H = H_{W_{1}} \circ (\pi_{1})_{|W_{1}}^{-1}.
\]
In addition, if ${\mathcal H}_{H}^{\Omega_{E}} \in
\Gamma({\mathcal T}^{E}E^{*})$ is the Hamiltonian section of $H$
with respect to the symplectic section $\Omega_{E}$ then $\xi_{1}$
and ${\mathcal H}_{H}^{\Omega_{E}}$ are $({\mathcal
T}(\pi_{1})_{|W_{1}}, (\pi_{1})_{|W_{1}})$-related, that is,
\[
{\mathcal T}(\pi_{1})_{|W_{1}} \circ \xi_{1} = {\mathcal
H}^{\Omega_{E}}_{H} \circ (\pi_{1})_{|W_{1}}.
\]
\end{remark}
\subsection{The variational point of view}\label{subsec:variational}

As it is well known, the dynamics in a vakonomic system is obtained through
the application of a constrained variational principle \cite{Ar}. More precisely,
let $Q$ be the configuration manifold, $L: TQ\to \R$ be the Lagrangian function
and $M\subset TQ$ be the constraint submanifold.
Then, one can consider the set of twice differentiable curves which connect two
arbitrary points $x,y\in Q$ as
\[
{\mathcal C}^2(x,y)=\{ a:[t_0,t_1]\to Q \, | \, a \mbox{ is } C^2, \, a(t_0)=x \mbox{ and }
a(t_1)=y\}.
\]
This is a smooth manifold of infinite dimension and, given $a\in {\mathcal C}^2(x,y)$,
the tangent space of ${\mathcal C}^2(x,y)$ is described as
\[
T_a{\mathcal C}^2(x,y)\kern-1pt=\kern-1pt\{ X\kern-2pt:[t_0,t_1]\to TQ \, | \, X \mbox{ is } C^1,\, X(t)\in T_{a(t)}Q ,
\, X(t_0)\kern-1.5pt=0 \mbox{ and } X(t_1)\kern-1.5pt=0\}.
\]
The elements of $T_a{\mathcal C}^2(x,y)$ can be seen as the infinitesimal variations
of the curve $a$. Next, one can introduce the action functional $\delta S:{\mathcal C}^2(x,y)\to \R$
defined by $a\mapsto \delta S(a) = \int _{t_0}^{t_1} L(\dot{a}(t))dt$.
Thus, the associated vakonomic problem $(Q,L,M)$ consists of extremizing the
functional $\delta S$ restricted to the subset $\widetilde{{\mathcal C}}^2(x,y)$,
which is given by
\[
\widetilde{{\mathcal C}}^2(x,y)=\{ a\in {\mathcal C}^2(x,y) \, |
\, \dot{a}(t)\in M_{a(t)}=M\cap \tau _Q^{-1}(a(t)), \,\forall\;
t\in [t_0,t_1] \},
\]
that is, $a\in \widetilde{{\mathcal C}}^2(x,y)$ is a solution of
the vakonomic system if $a$ is a critical point of $\delta
S_{|\widetilde{{\mathcal C}}^2(x,y)}$. In the particular case when
we do not have constraints, i.e. $M=TQ$, then we recover the
variational way to obtain Euler-Lagrange equations.

Now, consider a Lie algebroid $\tau :E\to Q$ and $L:E\to \R$ a
Lagrangian function on it. In \cite{mart2} it is shown how to
obtain the Euler-Lagrange equations (on the Lie algebroid) from a
variational point of view. Let us recall some aspects related with
this formulation. First, the set of $E$-paths is defined by
\[
{\mathcal Adm}([t_0,t_1],E)=\{ a:[t_0,t_1]\to E \, | \,  \rho\circ a=\frac{d}{dt}(\tau\circ a) \}.
\]
This set is a Banach submanifold of the set of $C^1$ paths in $E$ whose base path is $C^2$.
Two $E$-paths $a_0$ and $a_1$ are said to be $E$-homotopic if there
exists a morphism  of Lie algebroids
\[
\Phi : TI\times TJ \to E,
\]
where $I=[0,1]$, $J=[t_0,t_1]$, $a(s,t)=\Phi(\partial_t|_{(s,t)})$ and $b(s,t)=\Phi(\partial_s|_{(s,t)})$,
such that
\[
\begin{array}{ll}
a(0,t)=a_0(t),\qquad
&b(s,t_0)=0,\\
a(1,t)=a_1(t),\qquad
&b(s,t_1)=0.
\end{array}
\]
$E$-homotopy classes induce a second differentiable Banach manifold structure on
${\mathcal Adm}([t_0,t_1],E)$. The set of $E$-admissible paths with this second manifold
structure will be denoted by ${\mathcal P}([t_0,t_1],E)$. In addition, at each $E$-admissible
curve $a$, the tangent space is given in terms of the so-called complete lifts of sections.
In fact,
\[
T_a{\mathcal P}([t_0,t_1],E)=\{ \eta ^\textbf{c}\in T_a{\mathcal Adm}([t_0,t_1],E)\, | \,
\eta (t_0)=0 \quand \eta (t_1)=0\}.
\]
(for more details, see \cite{CraFer,mart2}).
We recall that if $\{ e_A\}$ is a local basis for $E$ and
$\eta$ is a time-dependent section locally given by
\[
\eta=\eta^A e_A
\]
then $\eta^\textbf{c}$, the complete lift of $\eta$, is the vector field on $E$ given by
\[
\eta^\textbf{c}=\eta^A \rho_A^i\displaystyle\frac{\partial }{\partial x^i} +
(\rho_B^i\displaystyle\frac{\partial \eta^C}{\partial x^i}-\eta^A
{\mathcal C}_{AB}^C)y^B\displaystyle\frac{\partial }{\partial y^C}.
\]
With the second manifold structure that it is introduced on the
space of $E$-paths, it is possible to formulate the variational
principle in a standard way. Let us fix two points $x,y\in M$ and
consider the set ${\mathcal P}([t_0,t_1],E)_{x}^{y}$ of $E$-paths
with f\/ixed base endpoints equal to $x$ and $y$, that is,
\[
{\mathcal P}([t_0,t_1],E)_{x}^{y}=\set{a\in{\mathcal P}([t_0,t_1],E)}{\tau(a(t_0))=x\quand \tau(a(t_1))=y}.
\]
In this setting, for the action functional $\delta S:{\mathcal P}([t_0,t_1],E)\to \R$ given by
\[
\delta S(a)=\int_{t_0}^{t_1} L(a(t))dt,
\]
the critical points of $\delta S$ on
${\mathcal P}([t_0,t_1],E)_{x}^{y}$ are the curves $a \in {\mathcal P}([t_0,t_1],E)_x^y$
which satisfy Euler-Lagrange equations \eqref{free-forces} (see \cite{mart2}).

Now, let $(L, M)$ be a vakonomic system on the Lie algebroid $\tau:
E \to Q$. We will denote by ${\mathcal P}([t_0,t_1],M)_x^y$ the set of $E$-paths in $M$
with fixed base endpoints equal to $x$ and $y$
\[
{\mathcal P}([t_0,t_1], M)_x^y =\big \{ a\in {\mathcal P}([t_0,t_1],E)_x^y \,|\, a(t)\in M, \,\forall
t\in [t_0,t_1] \big \}.
\]
We are going to consider infinitesimal variations (that is, complete
lifts  $\eta ^\textbf{c}$) tangent to the constraint submanifold $M$. We are going
to assume that there exist enough infinitesimal variations (that is, we are studying
the so-called normal solutions of the vakonomic problem). Since $M$
is locally given by $y^\alpha - \Psi ^\alpha (x^i, y^a)=0$, we
deduce that the allowed infinitesimal variations must satisfy
\[
\eta ^\textbf{c} (y^\alpha - \Psi ^\alpha (x^i, y^a))=0,\quad \eta (t_0)=0,\quad \eta (t_1)=0.
\]
Note that if $a\in {\mathcal P}([t_0,t_1],M)_x^y$ then
\[
\eta ^\textbf{c} (y^\alpha - \Psi ^\alpha (x^i, y^a))\circ a=0
\]
if and only if
\begin{equation}\label{relation-inf-var}
\frac{d\eta^\alpha}{dt}=\rho^i_A\eta ^A\frac{\partial \Psi
^\alpha}{\partial x^i}+\frac{d\eta^a}{dt}\frac{\partial \Psi
^\alpha}{\partial y^a}+{\mathcal C}^a_{BA}y^B\eta ^A\frac{\partial
\Psi ^\alpha}{\partial y^a}-{\mathcal C}^\alpha_{BA}y^B\eta ^A.
\end{equation}
Let us look for the critical points of the action functional
$\delta S:{\mathcal P}([t_0,t_1],E)^y_x\to \R$
\[
\begin{array}{rccl}
\delta S:& {\mathcal P}([t_0,t_1],E)^y_x &\to & \R \\
\displaystyle  & a(t)  & \mapsto &\displaystyle \int_{t_0}^{t_1} L(a(t))dt.
\end{array}
\]
If we consider our infinitesimal variations then,
\[
\begin{array}{rcl}
\displaystyle \frac{d}{ds}_{|s=0} \int_{t_0}^{t_1} L (a_s(t))dt &\kern-5pt=&\kern-5pt\displaystyle
\int_{t_0}^{t_1} \frac{\partial L}{\partial x^i}\eta^\textbf{c}_i +
\frac{\partial L}{\partial y^a}\eta^\textbf{c}_a + \frac{\partial
L}{\partial y^\alpha}\frac{\partial \Psi ^\alpha}{\partial x^i}
\eta^\textbf{c}_i +  \frac{\partial L}{\partial
y^\alpha}\frac{\partial \Psi ^\alpha}{\partial y^a} \eta^\textbf{c}_a dt \\
&\kern-5pt=&\kern-5pt\displaystyle \int_{t_0}^{t_1} \frac{\partial \tilde{L}}{\partial
x^i}\eta^\textbf{c}_i + \frac{\partial \tilde{L}}{\partial
y^a}\eta^\textbf{c}_a dt \\ &\kern-5pt=& \kern-5pt\displaystyle \int_{t_0}^{t_1} \frac{\partial
\tilde{L}}{\partial x^i}\rho ^i_\alpha \eta ^\alpha + \frac{\partial
\tilde{L}}{\partial x^i}\rho ^i_a \eta ^a + \frac{\partial
\tilde{L}}{\partial y^a}\eta^\textbf{c}_a dt
\end{array}
\]
Let $p_\alpha$ be the solution of the differential equations
\[
\dot{p}_{\alpha}=\Big ( \frac{\partial \tilde{L}}{\partial
x^i}-p_\beta \frac{\partial \Psi ^\beta}{\partial x^i} \Big ) \rho
^i_{\alpha} -y^a {\mathcal C}^B_{\alpha a}p_B-\Psi ^\beta {\mathcal
C}^B_{\alpha\beta}p_B,
\]
where
\begin{equation}\label{momenta-eq}
p_a=\frac{\partial \tilde{L}}{\partial y^a}-p_\alpha \frac{\partial
\Psi ^\alpha}{\partial y^a}.
\end{equation}
Using Eq. (\ref{relation-inf-var}), we get that
\[
\begin{array}{rcl}
\displaystyle \frac{d}{dt}( p_\alpha \eta ^\alpha)&=&
\displaystyle \dot{p}_\alpha \eta^\alpha + p_\alpha \dot\eta ^\alpha
\\[9pt] &=& \displaystyle p_\alpha \rho^i_a\eta ^a\frac{\partial \Psi
^\alpha}{\partial x^i}+p_\alpha\frac{d\eta^a}{dt}\frac{\partial \Psi
^\alpha}{\partial y^a}+p_\alpha{\mathcal C}^a_{BA}y^B\eta
^A\frac{\partial \Psi ^\alpha}{\partial y^a}-p_\alpha{\mathcal
C}^\alpha_{Ba}y^B\eta ^a \\[7pt] &&\displaystyle +\frac{\partial \tilde{L}}{\partial
x^i}\rho ^i_{\alpha}\eta ^\alpha - y^a \eta ^\alpha {\mathcal
C}^b_{\alpha a}p_b  -\eta ^\alpha \Psi ^\beta {\mathcal
C}^b_{\alpha\beta}p_b.
\end{array}
\]
If we use this equality, we deduce that
\[
\begin{array}{rcl}
\displaystyle \frac{d}{ds}_{|s=0} \int_{t_0}^{t_1} L (a_s(t))dt &=&\displaystyle
 \int_{t_0}^{t_1} \Big ( \frac{\partial \tilde{L}}{\partial x^i}\rho ^i_a \eta
^a + \frac{\partial \tilde{L}}{\partial y^a}\eta^\textbf{c}_a +y^a
\eta ^\alpha {\mathcal C}^b_{\alpha a}p_b + \eta ^\alpha \Psi ^\beta
{\mathcal C}^b_{\alpha\beta}p_b \\[7pt] && \displaystyle - p_\alpha \rho^i_a\eta
^a\frac{\partial \Psi ^\alpha}{\partial x^i}+\frac{d}{dt}\Big (
p_\alpha\frac{\partial \Psi ^\alpha}{\partial y^a}\Big ) \eta^a -
p_\alpha{\mathcal C}^d_{Ba}y^B\eta ^a\frac{\partial \Psi
^\alpha}{\partial y^d} \\[7pt] && \displaystyle - p_\alpha{\mathcal C}^d_{B\beta}y^B\eta
^\beta\frac{\partial \Psi ^\alpha}{\partial y^d} +p_\alpha{\mathcal
C}^\alpha_{Ba}y^B\eta ^a \Big )dt .
\end{array}
\]
Finally, using Eq. (\ref{momenta-eq}) and the fact that
\[
\eta^\textbf{c}_d=\frac{d\eta ^d}{dt}+ \Big ( {\mathcal
C}^d_{ba}\eta ^ay^b  + {\mathcal C}^d_{a\alpha}\eta ^\alpha y^a
+{\mathcal C}^d_{\beta a}\eta ^a\Psi ^\beta + {\mathcal
C}^d_{\beta\alpha}\eta ^\alpha\Psi^\beta \Big ),
\]
we obtain that
\[
\begin{array}{ll}
\displaystyle \frac{d}{ds}_{|s=0} \int_{t_0}^{t_1} L (a_s(t))dt =\\
\kern20pt \displaystyle
 \int_{t_0}^{t_1} \Big [
\Big ( \frac{\partial \tilde{L}}{\partial x^i}-p_\alpha
\frac{\partial \Psi ^\alpha}{\partial x^i} \Big )\rho ^i_{a}
-\frac{d}{dt} \Big ( \frac{\partial \tilde{L}}{\partial y^a}
-p_\alpha \frac{\partial \Psi ^\alpha}{\partial y^a} \Big ) -y^b
{\mathcal C}^B_{a b}p_B-\Psi ^\alpha {\mathcal C}^B_{a\alpha}p_B
\Big ]\eta ^a dt .
\end{array}
\]
Since the variations $\eta ^a$ are free, we conclude that the
equations are
\[
\left \{
\begin{array}{l}
\displaystyle \dot{x}^i=y^a \rho ^i_a+\Psi ^\alpha \rho ^i_\alpha ,\\[10pt]
\displaystyle \dot{p}_{\alpha}=\Big ( \frac{\partial
\tilde{L}}{\partial x^i}-p_\beta \frac{\partial \Psi
^\beta}{\partial x^i} \Big ) \rho ^i_{\alpha} -y^a {\mathcal
C}^B_{\alpha a}p_B-\Psi ^\beta {\mathcal
C}^B_{\alpha\beta}p_B ,\\[10pt]
\displaystyle \frac{d}{dt}\left( \frac{\partial \tilde{L}}{\partial
y^a}-p_\alpha \frac{\partial \Psi ^\alpha}{\partial y^a} \right)=
\Big ( \frac{\partial \tilde{L}}{\partial x^i}-p_\alpha
\frac{\partial \Psi ^\alpha}{\partial x^i} \Big ) \rho ^i_{a} -y^b
{\mathcal
C}^B_{a b}p_B-\Psi ^\alpha {\mathcal C}^B_{a\alpha}p_B,\\[10pt]
\end{array}
\right .
\]
with $p_a=\frac{\partial \tilde{L}}{\partial y^a}-p_\alpha
\frac{\partial \Psi^\alpha}{\partial y^a}$, that is, we obtain the
vakonomic equations for the vakonomic system $(L,M)$ on the Lie
algebroid $\tau :E\to Q$.

\subsection{Examples}\label{subsec:examples}
\subsubsection{Skinner-Rusk formalism on Lie algebroids} Suppose
that $\tau :E\to Q$ is a Lie algebroid, $L:E\to \R$ is a
Lagrangian function and $M=E$, that is, we do not have
constraints. Then,
\[
W_0=E^*\oplus E
\]
and the Pontryagin Hamiltonian $H_{W_0}:E^*\oplus E\to \R$ is
locally given by
\[
H_{W_0} (x^i,p_A,y^A)=y^Ap_A - L(x^i,y^A ).
\]
Let us apply the constraint algorithm. First, if $\Omega _0$ is the
presymplectic section on $W_0$ given by $\Omega _0= ({\mathcal T}\mathrm{pr}_1,
\mathrm{pr}_1)^*\Omega _E$ then, since $Ker\,\Omega _0=
span \{ {\mathcal V}_A \}$, the primary constraint submanifold $W_1$ is
locally characterized by
\begin{equation}\label{momenta}
p_A=\frac{\partial L}{\partial y^A},
\end{equation}
which are the definition of the momenta. $W_1$ is isomorphic to
$E$ and the vakonomic equations reduce to
\begin{equation}\label{Skinner}
\left \{
\begin{array}{l}
\displaystyle \dot{x}^i=y^A \rho ^i_A,\\[10pt]
\displaystyle \frac{d}{dt}\Bigl(\frac{\partial L}{\partial
y^A}\Bigr) = \rho_A ^i\frac{\partial L}{\partial x^i} - {\mathcal
C}_{AB}^C y^B p_C .
\end{array}
\right .
\end{equation}
Using (\ref{momenta}), it follows that Equations
(\ref{Skinner}) are just the Euler-Lagrange equations for $L$ (see
(\ref{free-forces})). We remark that if $E=TQ$, this procedure is
the Skinner-Rusk formulation of Lagrangian Mechanics (see
\cite{SR1,SR2}). On the other hand, using Proposition \ref{prop:charact-sympl},
we deduce that the vakonomic system is regular if and only if the Lagrangian
function $L$ is regular.

\subsubsection{The tangent bundle}
Let $E$ be the standard Lie algebroid $\tau _{TQ}:TQ\to Q$, $M\subseteq TQ$ be
the constraint submanifold such that $\tau _{TQ|M}:M\to Q$ is a surjective submersion
and $L:TQ\to \R$ be a standard Lagrangian function. Suppose that $(q^A,\dot{q}^A)=
(q^A,\dot{q}^a,\dot{q}^\alpha)$ are local fibred coordinates on $TQ$ and that the submanifold
$M$ is locally described by equations of the form
\[
\dot{q}^\alpha = \Psi ^\alpha (q^A,\dot{q}^a).
\]
As we know, the local structure functions of the Lie algebroid $\tau _{TQ}:TQ\to Q$
with respect to the local coordinates $(q^A,\dot{q}^A)$ are
\[
\rho_A^B=\delta _A^B \quand {\mathcal C}_{A B}^C =0.
\]
Thus, if we apply the results of Sections \ref{subsec:vak-bracket}
and \ref{subsec:variational} to this particular case we recover
the geometric formulation of vakonomic mechanics
developed in \cite{CLMM}. In particular, the vakonomic equations
reduce to
\[
\left \{
\begin{array}{l}
\displaystyle  \dot{q}^\alpha =\Psi ^\alpha, \\[10pt]
\displaystyle  \frac{d}{dt}\Big ( \frac{\partial
\tilde{L}}{\partial \dot{q}^a}\Big )- \frac{\partial \tilde{L}}{\partial
q^a}=\dot{p}_\alpha \frac{\partial \Psi ^\alpha}{\partial \dot{q}^a} +
p_\alpha \Big [ \frac{d}{dt}\Big ( \frac{\partial \Psi
^\alpha}{\partial \dot{q}^a}\Big ) - \frac{\partial \Psi
^\alpha}{\partial q^a} \Big ],\\[10pt]
\displaystyle \dot{p}_\alpha =\frac{\partial \tilde{L}}{\partial
q^\alpha }-p_\beta \frac{\partial \Psi ^\beta}{\partial q^\alpha}.
\end{array}
\right .
\]

\subsubsection{Lie algebras} Let $\mathfrak g$ be a real Lie algebra of
dimension $n$. Then, $\mathfrak g$ is a Lie algebroid over a single point.

Now, suppose that ${\mathfrak C}$ is an affine subspace of $\mathfrak g$ modelled over
the vector space $C$ of dimension $n-\bar{m}$ and that $e_0\in \mathfrak C$, $e_0\neq 0$.
We consider a basis $\{ e_A \}=\{ e_a,e_0,e_{\bar{a}}\}=\{ e_a,e_\alpha \}$
of $\mathfrak g$ such that $\{e_a \}$ is a basis of $C$ and
\[
[e_A,e_B] = {\mathcal C}_{AB}^Ce_C.
\]
Denote by $(y^a,y^0,y^{\bar{a}})=(y^a,y^\alpha )$ the linear coordinates on $\mathfrak g$
induced by the basis $\{e_a,e_0,e_{\bar{a}}\}=\{e_a,e_\alpha\}$. Then, $\mathfrak C$
is given by the equations
\[
y^0=1, \quad y^{\bar{a}}=0.
\]
Next, assume that $L:\mathfrak g\to \R$ is a Lagrangian function and denote
by $\tilde{L}:\mathfrak C\to \R$ the restriction of $L$ to $\mathfrak C$. Then,
a curve
\[
\sigma :t\mapsto (y^a (t),y^0(t),y^{\bar{a}}(t))=(y^a (t),1,0,\ldots, 0)
\]
in $\mathfrak C$ is a solution of the vakonomic equations for the constrained system
$(L,\mathfrak C)$ if and only if
\begin{equation}\label{Lavako}
\left\{\begin{array}{l} \displaystyle \frac{d}{dt}\Big (
\frac{\partial\tilde {L}}{\partial y^a}\Big ) = -
\frac{\partial\tilde {L}}{\partial y^c} (y^b {\mathcal C}^c_{ab}
+{\mathcal C}^c_{a0} )
- p_\beta ( y^b{\mathcal C}^\beta _{ab}+{\mathcal C}^\beta _{a0}),\\[10pt]
\displaystyle \dot{p}_\alpha =-\frac{\partial\tilde {L}}{\partial y^c}
( y^b{\mathcal C}^c_{\alpha b}+{\mathcal C}^c_{\alpha 0} )
- p_\beta ( y^b{\mathcal C}^\beta_{\alpha b}+{\mathcal C}^\beta _{\alpha 0} ).
\end{array}\right .
\end{equation}
Now, we consider the curve $\gamma$ in $\mathfrak g^*$ whose components with respect
to the dual basis $\{e^A \}$ are
\[
\gamma (t)=(\frac{\partial \tilde{L}}{\partial y^a}_{|\sigma (t)},p_\alpha (t)),
\]
that is, $\gamma$ is the sum of the curve $\frac{\partial \tilde{L}}{\partial y}$
and the curve
\[
t\mapsto \lambda (t) \in C^o
\]
whose components are $\lambda (t)=(0,p_\alpha (t))$. Here, $C^o\subseteq \mathfrak g^*$
is the annihilator of the subspace $C$. Then, a direct computation,
using \eqref{Lavako}, proves that $\gamma$ satisfies the Euler-Poincar\'e equations
\[
\frac{d}{dt}\Big ( \frac{\partial \tilde{L}}{\partial y} +\lambda
\Big )=ad ^*_\sigma \Big ( \frac{\partial \tilde{L}}{\partial y} +\lambda
\Big ).
\]
This result is just the ``\emph{Optimization Theorem for Nonholonomic Systems on
Lie groups}" which was proved in \cite{KM} (see Theorem 5.1 in \cite{KM}).

\subsubsection{Atiyah algebroids and reduction in subriemannian geometry} Let $\pi:P\to Q$ be a
principal bundle with structural group $G$. The dimension of $P$
(respectively, $Q$) is $n$ (respectively, $m$).

Suppose that $D$ is a distribution on $P$ such that
\[
T_pP=D_p+V_p\pi,\quad \mbox{ for all }p\in P,
\]
where $V\pi$ is the vertical bundle to the principal bundle
projection $\pi$.

We also suppose that $D$ is equipped with a bundle metric
$\langle\; ,\; \rangle_D$. Assuming that both $D$ and $\langle\;
,\; \rangle_D$ are $G$-invariant, then one may construct a
nonholonomic connection as follows (see \cite{BKMM}).

We will suppose that the space
\[
S_p=D_p\cap V_p\pi
\]
has constant dimension $r$, for all $p\in P$. Under this
condition, the horizontal space of the nonholonomic connection at
the point $p$ is $S_p^\perp$, where $S_p^\perp$ is the orthogonal
complement of $S_p$ on $D$ with respect to the bundle metric
$\langle\; ,\; \rangle_D$. We will denote by $\omega ^{nh}:TP\to
\mathfrak g$ the corresponding Lie algebra-valued 1-form.

Define a Lagrangian function on $D$ by $L(v_q)=\frac{1}{2}\langle
v_q, v_q \rangle_D$ where $v_q\in D_q$.

Now, we consider the Atiyah algebroid $\tau _P|G:TP/G\to Q=P/G$
associated with the principal bundle $\pi :P\to Q=P/G$. Note that
the space of orbits $D/G$ of the action of $G$ on $D$ is a vector
subbundle (over $Q$) of the Atiyah algebroid.

Next, we will obtain a local basis of $\Gamma (TP/G)$ adapted to
the vector subbundle $D/G$.

For this purpose, we choose a local trivialization of the
principal bundle $\pi :P\to Q=P/G$ to be $U\times G$, where $U$ is
an open subset of $Q$. Let $e$ be the identity element of $G$ and
assume that there are local coordinates $(x^i)$ on $U$. If
$(\frac{\partial }{\partial x^i})^h$ is the horizontal lift of the
vector field $\frac{\partial }{\partial x^i}$ on $U$ then
$(\frac{\partial }{\partial x^i})^h$ is a $G$-invariant vector
field and $\{ (\frac{\partial }{\partial x^i})^h \}_{i=1,\ldots
,m}$ is a local basis of $S^\perp$.

On the other hand, we consider a family of smooth maps
\[
\tilde{e}_A:U\to \mathfrak g, \quad A=1,\ldots ,n-m,
\]
such that, for every $q\in U$, $\{ \tilde{e}_A(q)\}$ is a basis of
$\mathfrak g$ and $\{ \tilde{e}_a(q)\}_{a=1,\ldots,r}$ is a basis
of the vector space
\[
\mathfrak g^{(q,e)}=\{\xi \in \mathfrak g\, | \, \xi _P (q,e)\in
D_{(q,e)}\}.
\]
Here, $\xi _P$ is the infinitesimal generator of the free action
of $G$ on $P$ associated with $\xi \in \mathfrak g$.

Now, we introduce the vertical vector fields on $U\times G$ given
by
\[
\begin{array}{rcll}
V_A:&U\times G &\to & T(U\times G)\cong TU\times TG \\
    & (q,g)    & \mapsto & \lvec{\tilde{e}_A(q)}(g)
\end{array}
\]
$\lvec{\tilde{e}_A(q)}$ is the left-invariant vector field on $G$
induced by the element $\tilde{e}_A(q)$ of $\mathfrak g$. Then, it
is clear that
\[
\{ (\frac{\partial }{\partial x^i})^h, V_A\} = \{ (\frac{\partial
}{\partial x^i})^h,V_a,V_\alpha \}
\]
is a local basis of $G$-invariant vector fields on $P$ and
\[
\{ (\frac{\partial }{\partial x^i})^h,V_a\}
\]
is a local basis of the space of sections of the vector subbundle
$D\to P$. Thus, we have the corresponding local basis of sections
\[
\{ e_i,e_A\}=\{ e_i,e_a,e_\alpha \}
\]
of the Atiyah algebroid $\tau _P|G: TP/G\to Q=P/G$ which is
adapted to the vector subbundle $D/G\to Q = P/G$.

Denote by $(x^i,\dot{x}^i,y^a, y^\alpha )$ the local coordinates
on $TP/G$ induced by the basis $\{ e_i,e_a,e_\alpha \}$. Then
$D/G$ is locally characterized by the equations
\[
y^\alpha =0.
\]
On the other hand, if
\[
\begin{array}{l}
\displaystyle \omega ^{nh} ( \frac{\partial }{\partial x^i} )=\Gamma ^A_i (q)e_A(q),\\[10pt]
\displaystyle \frac{\partial \tilde{e}_B }{\partial x^i}_{|q}=\chi
^C_{iB}(q)\tilde{e}_C(q),\quad [\tilde{e}_A(q),
\tilde{e}_B(q)]={\mathcal C}^C_{AB}\tilde{e}_C(q),
\end{array}
\]
for all $q\in U$, it follows that
\[
[e_i,e_j]={\mathcal B}^C_{ij}e_C,\quad [e_i,e_B]=\mu
^C_{iB}e_C,\quad [e_A,e_B]={\mathcal C}^C_{AB}e_C,
\]
where
\[
\begin{array}{l}
\displaystyle {\mathcal B}_{ij}^C=\frac{\partial \Gamma
^C_i}{\partial x^j} -\frac{\partial \Gamma ^C_j}{\partial
x^i}+\Gamma _i ^A \Gamma _j ^B{\mathcal C}_{AB}^C
+\Gamma _i^A\chi ^C_{jA}-\Gamma _j^A\chi ^C_{iA},\\[10pt]
\displaystyle \mu _{iB}^C=\chi ^C_{iB}-\Gamma _i^A{\mathcal C}
^C_{AB}.
\end{array}
\]
Moreover, if $\rho :TP/G\to TQ$ is the anchor map of the Atiyah
algebroid $\tau _P|G:TP/G\to Q=P/G$, we have that
\[
\rho (e_i)=\frac{\partial }{\partial x^i},\quad \rho (e_A)=0.
\]
Since the bundle metric $\langle\; ,\; \rangle_D$ is
$G$-invariant, it induces a bundle metric $\langle\; ,\;
\rangle_{D/G}$ on $D/G$, with associated lagrangian function
${l}:D/G\to \R$.  Therefore,  we deduce that a curve
\[
\sigma :t\mapsto (x^i(t),\dot{x}^i(t),y^a (t),y^\alpha (t))=
(x^i(t),\dot{x}^i(t),y^a (t),0)
\]
in $D/G$ is a solution of the vakonomic equations for the
constrained system $(l,D/G)$ if and only if
\[
\left\{
\begin{array}{l}
\displaystyle \frac{d}{dt}( \frac{\partial {l}}{\partial
\dot{x}^i})= \frac{\partial {l}}{\partial x^i}-(\dot{x}^j{\mathcal
B}_{ij}^a+y^b\mu _{ib}^a)\frac{\partial {l}}{\partial y^a} -
 (\dot{x}^j{\mathcal B}_{ij}^\alpha +y^b\mu _{ib}^\alpha )p_\alpha,\\[10pt]
\displaystyle \frac{d}{dt}( \frac{\partial {l}}{\partial y^a})=
(\dot{x}^j\mu _{ja}^c-y^b{\mathcal C}_{ab}^c)\frac{\partial
{l}}{\partial y^c}
+(\dot{x}^j\mu _{ja}^\alpha-y^b{\mathcal C}_{ab}^\alpha)p_\alpha ,\\[10pt]
\displaystyle \dot{p}_\alpha = (\dot{x}^i\mu
^b_{i\alpha}-y^a{\mathcal C}_{\alpha a}^b)\frac{\partial
{l}}{\partial y^b} +(\dot{x}^i\mu _{i\alpha}^\beta -y^a{\mathcal
C}_{\alpha a}^\beta )p_\beta .
\end{array}\right .
\]
Observe that if the basis of sections of $D/G$,  $\{e_i, e_a\}$,
is orthonormal  then
\[
l(x^i, \dot{x}^i, y^a)=\frac{1}{2}\left( \sum_i(\dot{x}^i)^2
+\sum_a (y^a)^2\right)
\]
and the vakonomic equations are:
\[
\left\{
\begin{array}{l}
\displaystyle{ \ddot{x}^i= -\sum_a(\dot{x}^j{\mathcal
B}_{ij}^a+y^b\mu _{ib}^a)y^a -
 (\dot{x}^j{\mathcal B}_{ij}^\alpha +y^b\mu _{ib}^\alpha )p_\alpha},\\[10pt]
\displaystyle{\dot{y}^a= \sum_c(\dot{x}^j\mu _{ja}^c-y^b{\mathcal
C}_{ab}^c) y^c
+(\dot{x}^j\mu _{ja}^\alpha-y^b{\mathcal C}_{ab}^\alpha)p_\alpha ,}\\[10pt]
\displaystyle \dot{p}_\alpha = \sum_b(\dot{x}^i\mu
^b_{i\alpha}-y^a{\mathcal C}_{\alpha a}^b)y^b +(\dot{x}^i\mu
_{i\alpha}^\beta -y^a{\mathcal C}_{\alpha a}^\beta )p_\beta .
\end{array}\right .
\]

Of course, it is also possible to study the abnormal solutions
taking the lagrangian $l\equiv 0$.

\begin{example} (See \cite{JurSha, Montgo}).
{\rm As a simple but illustrative example, consider $\R^3$ with
the distribution $D=\ker\omega$ where $\omega$ is the Martinet
1-form:  $\displaystyle{\omega=dx^3-\frac{(x^1)^2}{2} dx^2}$.
Consider the vector fields generating $D$:
\[
\frac{\partial}{\partial
x^2}+\frac{(x^1)^2}{2}\frac{\partial}{\partial x^3}, \quad
\frac{\partial}{\partial x^1},
\]
and the bundle metric  $\langle\; ,\; \rangle_D$ which makes both
vector fields  orthonormal.

Take now the action by translations:
\[
\begin{array}{rcl}
\R^2\times \R^3&\longrightarrow& \R^3\\
((a, b), (x^1, x^2, x^3))& \longmapsto& (x^1, x^2+a, x^3+b)\; .
\end{array}
\]
We have a principal bundle structure $\pi: \R^3\to \R^3/\R^2\equiv
\R$ being both  $D$ and $\langle\; ,\; \rangle_D$
$\R^2$-invariant.

Observe that $\displaystyle{V\pi=\hbox{span }\{
\frac{\partial}{\partial x^2}, \frac{\partial}{\partial x^3}\}}$,
 $S=D\cap V\pi=\hbox{span }\{\displaystyle{V_1=\frac{\partial}{\partial
x^2}+\frac{(x^1)^2}{2}\frac{\partial}{\partial x^3}}\}$ and
$S^\perp=\hbox{span
 }\{\displaystyle{\frac{\partial}{\partial x^1}} \}$.

 On the Atiyah algebroid $T\R^3/\R^2\cong \R\times \R^3\to\R$ we have the induced
 basis of sections $e_i:\R\to \R\times \R^3$, $i=1, 2, 3$:
 \begin{eqnarray*}
 e_1(x)&=& (x; (1, 0, 0)), \\
e_2(x)&=&(x; (0, 1, \frac{x^2}{2}))\\
e_3(x)&=&(x; (0, 0, 1))
\end{eqnarray*}
which induces coordinates $(x, \dot{x}, y^1, y^2)$.
 Then, $D/\R^2$ is
characterized by the equation $y^2=0$.

Observe that
\[
[e_1, e_2]=x e_3, \qquad [e_1, e_3]=0, \qquad [e_2, e_3]=0\; ,
\]
and
\[
l(x; \dot{x}, y^1)=\frac{1}{2}((\dot{x})^2+(y^1)^2).
\]
Consequently the vakonomic equations are:
\begin{eqnarray*}
\dot{p}_2&=&0\\
\ddot{x}&=& -y^1xp_2\\
\dot{y}^1&=& \dot{x}xp_2\; .
\end{eqnarray*}
That is, $p_2=k$ constant, and the vakonomic equations are
precisely,
\begin{eqnarray*}
\ddot{x}&=& -k y^1 x\\
\dot{y}^1&=& k\dot{x} x\; .
\end{eqnarray*}

Since $\frac{1}{2}(\dot{x}^2+ (y^1)^2)$ is a constant of motion
then we can take $\dot{x}(t)= r\sin\theta(t)$ and $y^1(t)=r\cos
\theta(t)$ where $\theta(t)$ verifies the equation of  pendulum
$\ddot\theta(t)=-k r \sin \theta(t)$ (see also Example
\ref{ball}).

 }

\end{example}

\subsubsection{Optimal Control on Lie algebroids as vakonomic
systems}(See \cite{CLMM2,reports}). Let $\tau :E\to Q$ be a Lie
algebroid and $C$ be a manifold fibred over the state manifold
$\pi :C\to Q$. We also consider a section $\sigma :C\to E$ along
$\pi$ and an index function $l:C\to \R$.

One important case happens when the section $\sigma :C\to E$ along
$\pi$ is an embedding, in this case, the image $M=\sigma(C)$ is a
submanifold of $E$. Moreover, since $\sigma: C\longrightarrow M$
is a diffeomorphism, we can define $L: M\longrightarrow \R$ by
$L=l\circ \sigma^{-1}$. In conclusion, it is equivalent to analyze
the optimal control defined by $(l, \sigma)$ (applying the
Pontryaguin maximum principle) that to study the vakonomic problem
on the Lie algebroid $\tau :E\to Q$ defined by $(L, M)$.

More generally (without assuming the embedding condition), we can
construct the prolongation $\tau ^\pi :{\mathcal T}^E C\to C$ of
the Lie algebroid $\tau:E\to Q$ over the smooth map $\pi :C\to Q$,
that is
\[
{\mathcal T}^EC = \{ (e,X_p) \in E_{\pi (p)} \times T_pC\, | \,
\rho (e) =T\pi (X_p) \}.
\]
Moreover, we have the constraint submanifold $M$ characterized
by
\[ M = \{ (e,X_p) \in {\mathcal T}^E_pC \, | \, \sigma (p)=e \}
\]
and the lagrangian function $L :{\mathcal T}^EC \to \R$ given by
$L = l\circ \tau ^\pi$. This is the \textit{vakonomic system
associated with the optimal control system}. If $E=TQ$ is the
tangent bundle of the state space $Q$, it is not difficult to show
that the prolongation of $TQ$ along $\pi :C\to Q$ is just the
tangent bundle $TC$. Under this isomorphism, the constraint
submanifold is
\[
M = \{ X \in TC \, | \, T\pi (X)=\sigma (\tau _C(X) ) \}.
\]
Thus, we recover the construction in \cite{CLMM2}, Section 4.

\begin{example}\label{ball}

{\rm Consider the following mechanical problem
\cite{CoLeMaMa,Koon,KM}. A (homogeneous) sphere of radius $r=1$,
mass $m$ and inertia about any axis $k^2,$ rolls without sliding
on a horizontal table which rotates with constant angular velocity
$\Omega$ about the $x^3$-axis. The coordinates of the point of
contact of the sphere with the plane are $(x^1, x^2)$.  The
configuration space of the sphere is $Q=\R^2\times SO(3)$ and the
Lagrangian of the system corresponds to the kinetic energy
\[
K(x^1, x^2; \dot{x}^1, \dot{{x}}^2, \omega_{x^1}, \omega_{x^2},
\omega_{x^3})=\frac{1}{2}(m(\dot{{x}}^1)^2+m(\dot{{x}}^2)^2 +
mk^2(\omega_{x^1}^2 + \omega_{x^2}^2 + \omega_{x^3}^2)),
\]
where $(\omega_{x^1}, \omega_{x^2}, \omega_{x^3})$ are the
components of the angular velocity of the sphere.

Since the ball is rolling without sliding on a rotating table then
the system is subjected to the affine constraints:
\[
\begin{array}{lcr}
\dot{{x}}^1-\omega_{x^2}=-\Omega {x^2},\\
\dot{{x}}^2 + \omega_{x^1}=\Omega {x^1},
\end{array}
\]
where $\Omega$ is constant. Moreover, it is clear that
$Q=\R^2\times SO(3)$ is the total space of a trivial principal
$SO(3)$-bundle over $\R^2$ and the bundle projection $\phi:Q\to
\R^2$ is just the canonical projection on the first factor.
Therefore, we may consider the corresponding Atiyah algebroid
$TQ/SO(3)$ over $\R^2$.

Since the Atiyah algebroid $TQ/SO(3)$ is isomorphic to the product
manifold $T\R^2\times {\mathfrak {so}}(3)\cong T\R^2\times \R^3$,
then a section of $TQ/SO(3)\cong T\R^2\times \R^3 \to \R^2$ is a
pair $(X,u)$, where $X$ is a vector field on $\R^2$ and $u:\R^2\to
\R^3$ is a smooth map. Therefore, a global basis of sections of
$T\R^2\times \R^3\to \R^2$ is
\[
e_1=(\displaystyle\frac{\partial}{\partial {x^1}},0),\,\,
e_2=(\displaystyle\frac{\partial}{\partial x^2},0),\,\,
e_3=(0,E_1),\,\, e_4=(0,E_2),\,\, e_5=(0,E_3).
\]
where $\{E_1, E_2, E_3\}$ is the canonical basis on $\R^3$.

The anchor map $\rho: T\R^2\times \R^3\to T\R^2$ is the projection
over the first factor and if $\lcf\cdot, \cdot \rcf$ is the Lie
bracket on the space $\Sec{TQ/SO(3)}$ then the only non-zero
fundamental Lie brackets are
\[
\lcf e_3,e_4\rcf=e_5,\;\;\;\lcf e_4,e_5\rcf=e_3,\;\;\; \lcf
e_5,e_3\rcf=e_4.
\]

It is clear that the Lagrangian and the nonholonomic constraints
are defined on the Atiyah algebroid $TQ/SO(3)$ (since the system
is $SO(3)$-invariant). In fact, we have a nonholonomic system on
the Atiyah algebroid $TQ/SO(3)\cong T\R^2\times \R^3.$ This kind
of systems was recently analyzed by J. Cort\'es {\it et al}
\cite{CoLeMaMa} (in particular, this example was carefully
studied).

After some computations the equations of motion for this
nonholonomic system are precisely
\begin{equation}\label{qwe}
\left.\begin{array}{rcl}
\dot{{x}}^1-\omega_{x^2}&=&-\Omega {x^2},\\
\dot{{x}}^2 + \omega_{x^1}&=&\Omega {x^1},\\
\omega_{x^3}&=&c
\end{array}
\right\}
\end{equation} where $c$ is a constant,
together with
\begin{eqnarray*}
\ddot{{x}}^1+\frac{k^2\Omega}{1+k^2}\dot{{x}}^2=0\\
\ddot{{x}}^2-\frac{k^2\Omega}{1+k^2}\dot{{x}}^1=0
\end{eqnarray*}

Now, we pass to an optimization problem. Assume full control over
the motion of the center of the ball (the shape variables) and
consider the cost function
\[
L({x^1}, {x^2}; \dot{{x}}^1, \dot{{x}}^2, \omega_{x^1},
\omega_{x^2}, \omega_{x^3})=\frac{1}{2}\left(
(\dot{{x}}^1)^2+(\dot{{x}}^2)^2\right)\; ,
\]
and the following optimal control problem:

{\bf PLATE-BALL PROBLEM} \cite{KM}. {\sl Given points $q_0, q_1\in
Q$, find an optimal control curve $({x^1}(t), {x^2}(t))$ on the
reduced space that steer the system from $q_0$ to $q_1$, minimizes
$\int_0^1 \frac{1}{2}\left(
(\dot{{x}}^1)^2+(\dot{{x}}^2)^2\right)\, dt$, subject to the
constraints defined by Equations (\ref{qwe}). }

Observe that the Plate-Ball problem is equivalent to the optimal
control problem given by the section $\sigma: \R^2\times \R^2\to
T\R^2\times \R^3$ along $T\R^2\times \R^3\to \R^2$ given by
\[
\sigma(x^1, x^2; u^1, u^2)=(x^1, x^2; u^1, u^2,-{u^2}+{\Omega
{x^1}}, {u^1}+{\Omega x^2},c)\, .
\]
and index function $l(x^1, x^2; u^1,
u^2)=\frac{1}{2}((u^1)^2+(u^2)^2)$. Since $\sigma$ is obviously an
embedding this is equivalent to the proposed Plate-Ball problem.

A necessary condition for optimality of the Plate-Ball problem is
given by the corresponding vakonomic equations.   For coherence
with the notation introduced along this paper, denote by
\[
y^1=\dot{x}^1, \ y^2=\dot{x}^2, \ y^3=\omega_{x^1},\
y^4=\omega_{x^2}, \ y^5=\omega_{x^3}
\]
Therefore, the vakonomic problem is given by the Lagrangian
$L=\frac{1}{2}\left( ({y^1})^2+({y^2})^2\right)$ and the
submanifold $M$ defined by the constraints:
\[
\begin{array}{rcl}
y^3&=&\Psi^3({x^1}, {x^2}, y^1, y^2)=-{y^2}+{\Omega {x^1}},\\
y^4 &=&\Psi^4({x^1}, x^2, y^1, y^2)={y^1}+{\Omega x^2},\\
y^5&=&\Psi^5({x^1}, x^2, y^1, y^2)=c
\end{array}
\]
After simple computations we obtain that the vakonomic equations
are:
\begin{eqnarray*}
\dot{p}_3&=& c p_4 -\left({y^1}+{\Omega
x^2}\right)p_5\\
\dot{p}_4&=& -c p_3 -\left({y^2}-{\Omega
{x^1}}\right)p_5\\
\dot{p}_5&=& \left({y^1}+{\Omega
x^2}\right)p_3+\left({y^2}-{\Omega
{x^1}}\right)p_4\\
\frac{d}{dt}\left(y^1-{p_4}\right)&=&-{\Omega}
p_3\\
\frac{d}{dt}\left(y^2+{p_3}\right)&=&-{\Omega}
p_4\\
y^1=\dot{x}^1\, , && y^2=\dot{x}^2\; .
\end{eqnarray*}

The system is obviously regular since the matrix
\[
\left( \frac{\partial^2 L}{\partial y^a\partial
y^b}-\sum_{\alpha=3}^5 p_{\alpha}\frac{\partial^2
\Psi^{\alpha}}{\partial y^a\partial y^b}\right)_{1\leq a,b\leq 2}
= \left(
\begin{array}{cc}
1&0\\
0&1
\end{array}
\right)
\]
is non-singular. Therefore, there exists a unique solution of the
vakonomic equations on $W_1$, determined by the conditions:
\begin{eqnarray*}
p_1&=&\frac{\partial L}{\partial y^1}-p_{\alpha}\frac{\partial \Psi^{\alpha}}{\partial y^1}=y^1-p_4\\
p_2&=&\frac{\partial L}{\partial y^2}-p_{\alpha}\frac{\partial
\Psi^{\alpha}}{\partial y^2}=y^2+p_3
\end{eqnarray*}
Therefore, it follows that $({x^1}, x^2, p_1, p_2, p_3, p_4, p_5)$
are local coordinates on $W_1$ (or, from Corollary 5.9, on $E^*$,
if you prefer).

Moreover, on $W_1$ we have a well-defined Poisson bracket $\{\; ,
\; \}_{(L, M)}$ whose non-vanishing terms are:
\begin{eqnarray*}
&\{{x^1}, p_1\}_{(L, M)}=1\; ,\  \{x^2, p_2\}_{(L, M)}=1&\\
&\{p_3, p_4\}_{(L, M)}=-p_5\; ,\  \{p_3, p_5\}_{(L, M)}=p_4\, ,\
\{p_4, p_5\}_{(L, M)}=-p_3&
\end{eqnarray*}
In these coordinates the Hamiltonian $H_{W_1}$ is:
\[
H_{W_1}({x^1}, x^2, p_1, p_2, p_3, p_4,
p_5)=\frac{1}{2}(p_1+p_4)^2+\frac{1}{2}(p_2-p_3)^2+cp_5+p_3\Omega
{x^1}+p_4\Omega x^2
\]
and the vakonomic equations are
\begin{eqnarray*}
\dot{p}_1&=& \{p_1, H_{W_1}\}_{(L, M)}=-{\Omega}
p_3\\
\dot{p}_2&=& \{p_2, H_{W_1}\}_{(L, M)}=-{\Omega}
p_4\\
\dot{p}_3&=& \{p_3, H_{W_1}\}_{(L, M)}= c p_4
-\left(p_1+p_4+{\Omega x^2}\right)p_5\\
\dot{p}_4&=& \{p_4, H_{W_1}\}_{(L, M)}= -c p_3
-\left(p_2-p_3-{\Omega
{x^1}}\right)p_5\\
\dot{p}_5&=& \{p_5, H_{W_1}\}_{(L, M)}= \left(p_1+p_4+{\Omega
x^2}\right)p_3+\left(p_2-p_3-{\Omega
{x^1}}\right)p_4\\
\dot{x}^1&=& \{{x^1}, H_{W_1}\}_{(L, M)}=p_1+p_4\\
\dot{x}^2&=& \{x^2, H_{W_1}\}_{(L, M)}=p_2-p_3\\
\end{eqnarray*}

One of the most studied cases is when $c=0$ (the sphere is rolled
that its angular velocity is always parallel to the horizontal
plane) and $\Omega=0$ (not rotation of the plane). In this case
the equations reduced to
\begin{eqnarray*}
&\dot{p}_3=-{y^1}p_5\; , \dot{p}_4= -{y^2}p_5\, ,\
\dot{p}_5= {y^1}p_3+{y^2}p_4&\\
&\frac{d}{dt}\left(y^1-{p_4}\right)=0\, ,
\frac{d}{dt}\left(y^2+{p_3}\right)=0&\\
&y^1=\dot{x}^1\, ,  y^2=\dot{x}^2\; .&
\end{eqnarray*}
{}From these equations it is easy to deduce that
\[
\frac{d}{dt}((y^1)^2+(y^2)^2)=0\;
\]
Therefore $y^1(t)=\cos \theta(t)$ and $y^2(t)=\sin \theta(t)$,
when $1=\sqrt{(y^1)^2+(y^2)^2}$. Moreover, it is easy to deduce
that $p_5=\dot{\theta}$.

Now, since $y^1=p_4+k_1$ and $y^2=-p_3+k_2$ with $k_1, k_2\in \R$
constants, taking $k_1=r\cos\varphi$ and $k_2=r\sin\varphi$ then
from Equation $\dot{p}_5= {y^1}p_3+{y^2}p_4$ we deduce that
\[
\ddot{\theta}=r\cos\theta\sin\varphi -
r\sin\theta\cos\varphi=-r\sin(\theta-\varphi)
\]
that is, the  angle $\theta$ satisfies the equation of pendulum,
while the  coordinates of the contact point satisfy the ODEs:
$\dot{{x}}^1=\cos \theta$ and $\dot{x}^2=\sin \theta$. The
remarkable result is that these equations say  that the contact
point of the sphere rolling optimally traces an Euler elastica
(see \cite{Jurdje}). }
\end{example}

\section{Conclusions and future work}
We have developed a general geometrical
setting for constrained mechanical systems in the context of Lie
algebroids.  We list the main results obtained in this paper:

\begin{itemize}
\item
We develop a constraint algorithm for presymplectic Lie algebroids
and discuss the reduction of presymplectic Lie algebroids.

\item For a singular Lagrangian function on a Lie algebroid,
we look for solutions of the corresponding dynamical equation,
applying the previously introduced constraint algorithm. In
addition, we find the submanifolds where the solution is a SODE (a
second order differential equation). The theory is illustrated
with an example.

\item We study vakonomic mechanics on Lie algebroids. As for
singular Lagrangian systems, we deduce the vakonomic equations by
means of a cons\-traint algorithm. We define the vakonomic bracket
in this setting. Furthermore, the variational point of view and
some explicit examples are discussed. We remark that we just look
for the so-called normal solutions. We postpone for a future work
a detailed variational analysis of vakonomic mechanics on Lie
algebroids, including abnormal solutions... We will illustrate our
theory with an example related with Optimal Control Theory, but
our approach may cover other interesting examples (see
\cite{Bl,HusBlo}) and admits, in a standard way, variational
discretizations (see \cite{BM,IMMM}).

\end{itemize}

\section*{Acknowledgments}
This work has been partially supported by MEC (Spain) Grants MTM
2006-03322, MTM 2007-62478, project ``Ingenio Mathematica"
(i-MATH) No. CSD 2006-00032 (Consolider-Ingenio 2010) and
S-0505/ESP/0158 of the CAM. D. Iglesias wants to thank MEC for a
Research Contract ``Juan de la Cierva".

\end{document}